\documentclass[12pt]{amsart}
\usepackage{amsmath,amssymb,amsthm,tikz,array,mathtools,tikz-cd}

\usepackage{hyperref}

\newtheorem{thm}{Theorem}[section] 
\newtheorem*{thm*}{Theorem}

\newtheorem{conj}[thm]{Conjecture}
\newtheorem*{conj*}{Conjecture}
\newtheorem{cor}[thm]{Corollary}
\newtheorem{defn}[thm]{Definition}
\newtheorem{exmpl}[thm]{Example}
\newtheorem{lem}[thm]{Lemma}
\newtheorem{prop}[thm]{Proposition}
\newtheorem{rem}[thm]{Remark}
\newtheorem{ques}[thm]{Question}

\renewcommand{\Pr}{\mathbb{P}}
\newcommand{\Z}{\mathbb{Z}}

\newcommand{\R}{\mathbb{R}}
\newcommand{\C}{\mathbb{C}}
\newcommand{\F}{\mathbb{F}}
\newcommand{\id}{\mathrm{id}}
\newcommand{\sg}[1]{\left\langle #1\right\rangle }

\newcommand{\set}[1]{\left\{ #1\right\}}
\newcommand{\sub}{\subseteq}

\newcommand{\normal}{\unlhd}
\newcommand{\stab}{\operatorname{stab}}
\newcommand{\eps}{\varepsilon}
\newcommand{\Ber}{\mathrm{Ber}}
\newcommand{\suchthat}{\mid}
\newcommand{\mul}[1]{#1^{\times}}
\newcommand{\tr}{\operatorname{tr}}
\newcommand{\mixlen}{\ell_{\mathrm{mix}}}

\newcommand\floor[1]{\left\lfloor #1\right\rfloor}
\newcommand{\Sym}[1]{\mathrm{Sym}({#1})}
\newcommand{\Alt}[1]{\mathrm{Alt}({#1})}
\newcommand{\Dih}[1]{\mathrm{Dih}({#1})}
\newcommand{\Cox}[2]{\mathsf{#1}_{#2}}
\def\eg{{e.g. }}
\newcommand\Cref[1]{{Corollary~\ref{#1}}}
\newcommand\Cjref[1]{{Conjecture~\ref{#1}}}
\newcommand\Eref[1]{{Example~\ref{#1}}}
\newcommand\Lref[1]{{Lemma~\ref{#1}}}
\newcommand\Pref[1]{{Proposition~\ref{#1}}}
\newcommand\Qref[1]{{Question~\ref{#1}}}
\newcommand\Rref[1]{{Remark~\ref{#1}}}
\newcommand\Sref[1]{{Section~\ref{#1}}}
\newcommand\Tref[1]{{Theorem~\ref{#1}}}

\newcommand\M[1][d]{{\operatorname{M}_{#1}}}
\newcommand\GL[1][d]{{\operatorname{GL}_{#1}}}
\newcommand\PGL[1][d]{{\operatorname{PGL}_{#1}}}
\newcommand\SL[1][d]{{\operatorname{SL}_{#1}}}
\newcommand\PSL[1][d]{{\operatorname{PSL}_{#1}}}
\newcommand\Zent{{\operatorname{Z}}}
\newcommand\wt{{\omega}}
\newcommand\card[1]{{\left|{#1}\right|}}
\newcommand\dimcol[2]{{[{#1}\,{:}\,{#2}]}}
\newcommand\HS{{\operatorname{HS}}}
\newcommand\normali{{\,\lhd\,}}
\newcommand\innprod[2]{{\left<#1,#2\right>}}

\newcommand\relem[2]{{\left(#2\!:\,#1\right)}}
\def\One{{\boldsymbol{1}}}
\def\twogen{{$\set{2}$-generated}}

\title{Mixability of finite groups}

\author{Gideon Amir}
\address{Bar-Ilan University, Ramat Gan 5290002, Israel}
\email{gideon.amir@biu.ac.il}
\author{Guy Blachar}
\address{Bar-Ilan University, Ramat Gan 5290002, Israel}
\email{guy.blachar@gmail.com}
\author{Subhajit Ghosh}
\address{Indian Institute of Technology Madras, Chennai 600036, India}
\email{gsubhajit@alum.iisc.ac.in}
\author{Uzi Vishne}
\address{Bar-Ilan University, Ramat Gan 5290002, Israel}
\email{vishne@math.biu.ac.il}
\thanks{UV and GB acknowledge support by an Israel Science Foundation grant \#1994/20. GA acknowledges support by an Israel Science Foundation grant \#957/20. SG acknowledges support from an Inspire Faculty Fellow research grant (IFA 23 MA 198).}
\makeatletter
\@namedef{subjclassname@2020}{\textup{2020} Mathematics Subject Classification}
\makeatother

\begin{document}

\begin{abstract}
Say that a finite group $G$ is mixable if a product of random elements, each chosen independently from two options, can distribute uniformly on $G$. We present conditions and obstructions to mixability. We show that $2$-groups, the symmetric groups, the simple alternating groups, several matrix and sporadic simple groups, and most finite Coxeter groups, are mixable. We also provide bounds on the mixing length of such groups.
\end{abstract}
\maketitle

\section{Introduction}\label{intro}

Card shuffling, and more generally the study of mixing times of random walks on groups, has attracted considerable attention in the past decades. This has led to the development of powerful tools for the analysis of Markov chains on groups, including connections to representation theory, Fourier analysis, and spectral graph theory.

In the current paper we introduce a new notion of group mixing, which we call \emph{mixability} of groups. We say that a finite group~$G$ is {\nobreak\textbf{mixable}} if there
are pairs of elements $a_{i}^{(0)}, a_{i}^{(1)} \in G$, and independent Bernoulli variables $\eps_1,\dots,\eps_k$, such that the product $a_{1}^{(\eps_1)}\cdots a_{k}^{(\eps_k)}$ distributes uniformly on $G$.
We further define the \textbf{mixing length} of~$G$, denoted $\mixlen(G)$, to be the minimal length of such a random product. Similar to random walks, here the randomness comes from choosing the steps at random. We note immediately that by conjugating, we may assume each $a_{i}^{(0)} = 1$ (see \Rref{samedef}).
Writing $a_i = a_i^{(1)}$, the steps become fixed and the randomness of the product
$g_1^{\eps_1}\cdots g_k^{\eps_k}$
comes from choosing whether or not to take each step.

\medskip

There are numerous works on mixing a group via random walks. This goes back to a work of Erd\H{o}s and R\'{e}nyi~\cite{ErdosRenyi65}, who studied random products of the above form where $\eps_1,\dots,\eps_k$ are assumed to be i.i.d.\ $\Ber(\frac{1}{2})$ random variables. Their work shows that when $G$ is abelian and one chooses the elements $g_1,\dots,g_k$ at random, then the distribution of the product $g_1^{\eps_1} \cdots g_k^{\eps_k}$ is approximately uniform once $k\ge 2\log_2\left|G\right|+O(1)$. We adopt their terminology, and refer to products of the form $g_1^{\eps_1}\cdots g_k^{\eps_k}$ as \textbf{random subproducts}, even when the variables $\eps_1,\dots,\eps_k$ are not uniform. Approximation of the distribution is not without flaws; for example it was noticed in \cite{UV98} that a group can be mixed in this sense with some conjugacy classes never touched. What determines the nature of mixability is that we ask the distribution of the random subproduct to be {\it{precisely}} uniform on~$G$.

Our work is largely inspired by a paper of Angel and Holroyd \cite{AH}. Their paper studies mixings of the symmetric group $\Sym{n}$ by ``lazy transpositions'', i.e., the mixability and mixing length of $\Sym{n}$ when the elements $g_1,\dots,g_k$ are taken to be transpositions. The authors present several mixings of $\Sym{n}$ of length $\binom{n}{2}$ using transpositions, and show that this is optimal if one uses only adjacent transpositions. A paper of Groenland, Johnston, Radcliffe and Scott~\cite{GroenlandJohnstonRadcliffeScott22} improves the above upper bound to $\le\frac{2}{3}\binom{n}{2}+O(n\log n)$, which is the best upper bound known so far. The only known lower bound on the length of such a mixing is $\log_2\left|\Sym{n}\right|\approx n\log_2 n$, coming from a trivial counting argument.

A refinement of the mixability of $\Sym{n}$ has also been studied by Janzer, Johnson and Leader~\cite{JanzerJohnsonLeader23}. They investigate the length of partial mixings of $\Sym{n}$ using lazy transpositions, where the mixing only needs to mix the action of $\Sym{n}$ on $k$-tuples rather than be a random uniform permutation.\medskip

In the current paper, we are concerned with the mixability of arbitrary finite groups. Our main question is the following:

\begin{ques}\label{ques:main}
  Which finite groups are mixable?
\end{ques}

We show in \Sref{sec:basic-defs} that the class of mixable groups is closed under tanking quotients and extensions. We demonstrate throughout the paper many examples of mixable groups, including the following:

\begin{thm*}
  The following families of finite groups are mixable:
  \begin{enumerate}
    \item\label{item:main-thm1} $2$-groups (\Cref{cor:2-groups}).
    \item\label{item:main-thm2} $\Sym{n}$ (\Eref{exmpl:Sn-mix}).
    \item\label{item:main-thm3} $\Alt{n}$ for $n\ge 5$ (\Cref{cor:alternating}).
        \item\label{item:main-thm6} All finite irreducible Coxeter groups, with the possible exception of those of types $\Cox{E}{6}$, $\Cox{E}{7}$ and $\Cox{E}{8}$ (\Tref{thm:coxeter}).
    \item\label{item:main-thm4}
    \begin{enumerate}
        \item $\PSL[2](\F_{2^n})$ for all $n\ge 1$ (\Tref{thm:PSL2-F2n}).
        \item $\SL[d](\F_q)$, $\PSL[d](\F_q)$ and $\PGL[d](\F_q)$ if $q-1$ is a power of $2$ (\Cref{cor:matrix-all-d}).
    \item $\GL[d](\F_q)$ if and only if $q-1$ is a power of $2$ (\Cref{cor:GLd}).
    \end{enumerate}
    \item\label{item:main-thm5} The Mathieu groups $M_{11},M_{12},M_{22},M_{23},M_{24}$ (\Cref{cor:mathieu}) and the Higman--Sims group $\HS$ (\Cref{cor:higman-sims}).
  \end{enumerate}
\end{thm*}

Conversely, we show:

\begin{thm*}[{\Tref{thm:odd-order}}]
  A finite group with a nontrivial quotient of odd order is not mixable.
\end{thm*}

We prove mixability by several techniques. The mixability of $2$-groups follows directly by considering their composition series. We introduce a notion of mixability for group actions that generalizes group mixability, and use it to present explicit mixings of the symmetric and alternating groups.

Next, we study mixability using representation theory. A finite group~$G$ is mixable if and only if in the group algebra $\C[G]$, the central idempotent $\frac{1}{\left|G\right|}\sum_{g\in G}g$ can be decomposed as a product of the form
\[
    \frac{1}{\left|G\right|}\sum_{g\in G}g = \prod_{i=1}^k((1-p_i)e_G+p_ig_i)
\]
for some $g_1,\dots,g_k\in G$ and $0\le p_1,\dots,p_k\le 1$ (see \Sref{sec:rep-theory} for further details). Special decompositions of this idempotent have previously been used to study questions of similar flavour, such as in the work of Diaconis and Shahshahani \cite{DiaconisShahshahani86}. This approach provides an alternate proof of \Tref{thm:odd-order} (see \Cref{cor:reps--1-ev} and the discussion afterwards), as well as to study the mixability of groups with low-dimensional representations, such as the dihedral groups $\Dih{n}$. We also investigate the mixability of finite Coxeter groups using their classification and a combination of the above approaches.\medskip

We show that when a group acts $2$-transitively on a set, and the point stabilizer is mixable, then the whole group is mixable (\Cref{cor:2-transitive}). We use this to prove that $\SL[d](\F_q)$ and $\PSL[d](\F_q)$ are mixalbe when $q-1$ is a power of~$2$, as well as the Mathieu groups and the Higman--Sims group.

We conjecture that \Tref{thm:odd-order} is the only obstruction to mixability:

\begin{conj*}[{\Cjref{conj:mix-quotients}}]\label{conj:main-conj}
  A finite group $G$ is mixable if and only if it has no nontrivial quotients of odd order.
\end{conj*}

In particular, we conjecture that the finite simple groups are all mixable.
Evidence supporting this conjecture comes from analyzing low-dimensional real representations of groups.\medskip

Other than studying which groups are mixable, we are also interested in estimating the mixing length of the mixable groups. For any finite group $G$ the entropy bound $$\mixlen(G)\ge\log_2\left|G\right|,$$ follows from any random subproduct of length $k$ being supported on at most $2^k$ values (\Rref{rem:entropy-bound}). We show that this bound is tight for $2$-groups (\Cref{cor:2-groups}), and that it is tight for $\Sym{n}$ and $\Alt{n}$ up to a factor of $\frac{3}{2}$ (see \Tref{cor:mixlen-sym} and \Cref{cor:alternating}). These examples lead us to ask:

\begin{ques}\label{ques:mixlen-bound}
  Is the bound $\mixlen(G)\ge\log_2\left|G\right|$ tight up to some universal constant? That is, does there exist a constant $C>0$ such that, for any finite mixable group $G$, one has $\mixlen(G)\le C\log_2\left|G\right|$?
\end{ques}

One interesting case is the mixing length of the dihedral groups~$\Dih{n}$. We show that if $n=2^tm$ with $m$ odd, then $\mixlen(\Dih{n})\le t+m$, by studying their representations (see \Cref{cor:dihedral}). As the trivial lower bound is $t+\log_2 m+1$, this is a good test case for \Qref{ques:mixlen-bound}.

\subsection*{Structure of the paper.}

In \Sref{sec:basic-defs} we study basic properties of mixability of groups and group actions. We use these ideas in \Sref{sec:basic-examples} to show that $2$-groups are mixable, while nontrivial groups of odd order are not. We then deduce that an abelian group is mixable if and only if it is a $2$-group.

\Sref{sec:perm-groups} is dedicated to study the symmetric and alternating groups. We show that they are mixable (other than $\Alt{3}$ and $\Alt{4}$) by providing an explicit mixing sequence, and that their mixing length is close to optimal.

We study the connection of mixability to representation theory in \Sref{sec:rep-theory}, showing that it is enough to check the mixability of each irreducible representation of the group. In \Sref{sec:struc} we define a group to be $2'$-simple if it has no odd order quotients, and \twogen\ if it is generated by involutions. We prove that
a finite group $G$ is $2'$-simple if and only if it has a subnormal series with \twogen\ quotients. This allows us to reduce \Cjref{conj:mix-quotients} to \twogen\ groups, which is, admittedly, a very large family, including by Feit-Thompson all simple groups.
In \Sref{sec:8} we provide evidence for this conjecture by studying low-dimensional real representations. In \Sref{sec:coxeter} we show that most finite Coxeter groups are mixable, combining the above ideas and following the classification of finite Coxeter groups.

The proof that $\Sym{n}$ is mixable motivates an approach to mixability of $2$-transitive groups, which we study in \Sref{sec:2-transitive}.
We apply our criterion to several examples of finite simple groups. We conclude by studying classical examples, namely matrix groups and the small sporadic simple groups in \Sref{sec:10}~and~\ref{sec:11}.

\section{Mixability of groups and of group actions}\label{sec:basic-defs}

We now define the notions of mixability for finite groups and for group actions, and study their basic properties.

\begin{defn}\label{def:mixability}
  Let $G$ be a finite group.
  \begin{enumerate}
    \item A \textbf{random subproduct} in $G$ is a product $g_1^{\eps_1}\cdots g_k^{\eps_k}$, where $g_1,\dots,g_k\in G$ are fixed elements, and $\eps_1,\dots,\eps_k$ are independent Bernoulli random variables $\eps_i\sim\Ber(p_i)$.
    \item We say that $G$ is \textbf{mixable} if there exists a random subproduct $g_1^{\eps_1}\cdots g_k^{\eps_k}$ in $G$ that is distributed uniformly on $G$. In this case, we say that $\relem{g_1}{p_1},\dots,\relem{g_k}{p_k}$ is a \textbf{mixing sequence} of~$G$.
    \item We define the \textbf{mixing length} of $G$, denoted $\mixlen(G)$, as the minimal length of a mixing sequence of $G$.
  \end{enumerate}
\end{defn}

We sometimes refer to a mixing sequence as a mixing algorithm for~$G$, i.e., sequentially applying the elements $g_1,\dots,g_n$ with probabilities $p_1,\dots,p_n$.

In the introduction we defined mixability slightly differently, via random choice between pairs of elements.
\begin{rem}\label{samedef}
Mixability via pairs and mixability via random subproducts are equivalent notions, with the same length function.  Indeed if $x_i$ is chosen by a Bernoulli variable $\eps_i$ from $\set{a_i,b_i}$, then
\begin{eqnarray*}
x_1\cdots x_n & = & (b_1a_1^{-1})^{\eps_1}(a_1b_2a_2^{-1}a_1^{-1})^{\eps_2}
(a_1a_2b_3a_3^{-1}a_2^{-1}a_1^{-1})^{\eps_3} \cdots \\
& & \qquad \cdot (a_1\cdots a_{n-1}b_na_n^{-1}\cdots a_1^{-1})^{\eps_n}a_1\cdots a_n,
\end{eqnarray*}
which is a random subproduct times a constant element.
\end{rem}

\begin{rem}\label{rem:entropy-bound}
  For any finite group $G$ we have $\mixlen(G)\ge\log_2\left|G\right|$, since the support of any random subproduct of length $k$ is of size at most $2^k$.
\end{rem}

We give two examples of mixable groups with explicit mixing sequences: cyclic $2$-groups and the symmetric groups.

\begin{exmpl}\label{exmpl:Z-mod-pow2}
  The cyclic group $\Z/2^d\Z$ is mixable for all $d\ge 1$. Indeed, the following is a mixing sequence of $\Z/2^d\Z$:
  \[
\relem{1}{{\textstyle\frac{1}{2}}}, \relem{2}{{\textstyle\frac{1}{2}}}, \relem{4}{\textstyle\frac{1}{2}}, \dots, \relem{2^{d-1}}{\textstyle\frac{1}{2}}.
  \]
\end{exmpl}

\begin{exmpl}\label{exmpl:Sn-mix}
  As shown by Angel and Holroyd \cite{AH}, and was already observed by Hall, Puder and Sawin \cite[Remark~4.7]{HallPuderSawin18}, the symmetric group $\Sym{n}$ is mixable for all $n\ge 1$. Indeed, the following algorithm mixes $\Sym{n}$ inductively:
  \begin{enumerate}
      \item Mix $\Sym{n-1}\le \Sym{n}$.
      \item Mix the position of $n$ by applying $(i-1,i)$ with probability $\frac{i-1}{i}$ for $i=n,n-1,\dots,2$.
  \end{enumerate}
  The induced mixing sequence is
  \[
    \relem{(1,2)}{{\textstyle\frac{1}{2}}},\relem{(2,3)}{\textstyle\frac{2}{3}},
    \relem{(1,2)}{{\textstyle\frac{1}{2}}},\relem{(3,4)}{\textstyle\frac{3}{4}},
    \relem{(2,3)}{{\textstyle\frac{2}{3}}},\relem{(1,2)}{\textstyle\frac{1}{2}},\dots
  \]
  One easily sees that the length of this mixing is $\binom{n}{2}$. In \Sref{sec:perm-groups} we will show that there exists a mixing of $\Sym{n}$ of length $\frac{3}{2}\log_2(n!)+\frac{1}{2}n$ which uses involutions rather than only transpositions.
\end{exmpl}

We extend the notion of mixability to group actions:

\begin{defn}
  Let $G$ be a finite group acting transitively on a set~$X$. We say that the action of $G$ on $X$ is \textbf{mixable} if there exists a random subproduct $g_1^{\eps_1}\cdots g_k^{\eps_k}$ in $G$ such that $g_1^{\eps_1}\cdots g_k^{\eps_k}x_0$ is distributed uniformly on $X$, where $x_0 \in X$ is some fixed point. In this case, we say that $\relem{g_1}{p_1},\dots,\relem{g_k}{p_k}$ is a \textbf{mixing sequence} of the action with respect to $x_0$. The minimal length of a mixing sequence is the \textbf{mixing length} of the action.
\end{defn}

By transitivity, if an action of $G$ on $X$ is mixable with respect to some $x_0\in X$, then it is mixable with respect to every $x\in X$. Indeed, let~$x\in X$, and fix $g\in G$ such that $gx_0=x$. If $g_1^{\eps_1}\cdots g_k^{\eps_k}x_0$ is uniform on $X$, then $(gg_1g^{-1})^{\eps_1}\cdots (gg_kg^{-1})^{\eps_k}x$ is also uniform on $X$. Thus, mixability is a property of the group action, regardless of the base point.

This notion indeed generalizes the definition of mixability of groups:

\begin{rem}
    A group $G$ is mixable if and only if the regular action of~$G$ on itself is mixable.
\end{rem}

For the converse, we begin with the following lemma:

\begin{lem}\label{lem:act-of-uniform}
  Suppose that $G$ acts transitively on a set $X$. If $\boldsymbol{g}$ is a random variable distributed uniformly on $G$, then for any $x\in X$, the random variable $\boldsymbol{g}x$ is distributed uniformly on $X$.
\end{lem}

\begin{proof}
  Let $y\in X$. By transitivity, there exists some $h\in G$ such that $hx=y$, and thus
  \[
    \Pr(\boldsymbol{g}x=y)=\Pr(\boldsymbol{g}\in h\cdot\stab_G(x))=\frac{\left|\stab_G(x)\right|}{\left|G\right|}=\frac{1}{\left|Gx\right|}=\frac{1}{\left|X\right|},
  \]
  where the last equality also follows from transitivity.
\end{proof}

As a direct consequence,

\begin{cor}
  If $G$ is mixable, then any transitive action of $G$ is mixable.
\end{cor}

We turn to study the relation between mixability of a group and the mixability of its subgroups and quotients. We will later show (\Tref{thm:odd-order}) that $\Alt{3}$ is not mixable, even though it is a subgroup of the mixable group $\Sym{3}$. This shows that mixability does not pass to subgroups, and not even to normal subgroups.\medskip

To study quotient groups, we use the following:

\begin{prop}
  Let $G$ be a group, and let $N\normali G$ be a normal subgroup. Then the group $G/N$ is mixable if and only if the natural action of $G$ on $G/N$ is mixable, in which case $\mixlen(G/N)=\mixlen(G,G/N)$.
\end{prop}

\begin{proof}
  First note that for any $g_1,\dots,g_k\in G$ and Bernoulli random variables $\eps_1,\dots,\eps_k$,
  \[
    (g_1^{\eps_1}\cdots g_k^{\eps_k})N=(g_1N)^{\eps_1}\cdots(g_kN)^{\eps_k}.
  \]
  The group $G/N$ is mixable if and only if there is some choice of $g_1,\dots,g_k$ and $\eps_1,\dots,\eps_k$ such that the right hand side is uniform on~$G/N$, whereas the action of $G$ on $G/N$ is mixable (with respect to $N$) if and only if there is some choice $g_1,\dots,g_k$ and $\eps_1,\dots,\eps_k$ so the left hand side is uniform on $G/N$. Therefore, these notions are equivalent.
\end{proof}

\begin{cor}\label{cor:quotient}
Let $N\normali G$ be a normal subgroup of~$G$.
If $G$ is mixable, then $G/N$ is mixable. Furthermore, $\mixlen(G/N)\le\mixlen(G)$.
\end{cor}

For the converse, we use the following:

\begin{prop}\label{prop:H-and-G/H}
  Let $G$ be a finite group, and let $H\le G$ be a subgroup of $G$. If $H$ is mixable and the action of $G$ on $G/H$ is mixable, then~$G$ is mixable and
  \[
    \mixlen(G)\le\mixlen(H)+\mixlen(G,G/H).
  \]
\end{prop}

\begin{proof}
  Take $g_1,\dots,g_k\in G$ and Bernoulli random variables $\eps_1,\dots,\eps_k$ such that $\boldsymbol{g}=g_1^{\eps_1}\cdots g_k^{\eps_k}$ mixes the action of $G$ on~$G/H$. Therefore the random coset $\boldsymbol{g}H$ is uniform on the set $G/H$. By assumption,~$H$ is mixable as well, so there are $h_1,\dots,h_m\in H$ and independent Bernoulli random variables $\delta_1,\dots,\delta_m$ such that the random subproduct $\boldsymbol{h}=h_1^{\delta_1}\cdots h_m^{\delta_m}$ is uniform on $H$.

  We claim that the random subproduct $\boldsymbol{g}\boldsymbol{h}$ is uniform on $G$. Indeed, for any $a\in G$,
  \begin{align*}
    \Pr(\boldsymbol{g}\boldsymbol{h}=a) & = \sum_{g\in G}\Pr(\boldsymbol{h}=g^{-1}a)\Pr(\boldsymbol{g}=g) = \sum_{g\in aH}\Pr(\boldsymbol{h}=g^{-1}a)\Pr(\boldsymbol{g}=g) \\
    & = \frac{1}{\left|H\right|}\sum_{g\in aH}\Pr(\boldsymbol{g}=g) = \frac{1}{\left|H\right|}\Pr(\boldsymbol{g}\in aH) = \frac{1}{\left|H\right|}\frac{1}{{[G:H]}}=\frac{1}{\left|G\right|}.
  \end{align*}
\end{proof}

It now follows that:

\begin{cor}\label{cor:N-and-G/N}
  Let $N\normali G$ be a normal subgroup. If $N$ and $G/N$ are mixable, then $G$ is mixable. Furthermore
  \[
    \mixlen(G)\le\mixlen(N)+\mixlen(G/N).
  \]
\end{cor}

We also note the following refinement, which shows that the mixability of an action passes to quotient spaces.
\begin{prop}\label{prop:G/K-and-G/H}
  Let $G$ be a finite group, and let $K \le H \le G$ be subgroups. If the action of $G$ on $G/K$ is mixable, then the action of $G$ on $G/H$ is mixable, and
  \[
    \mixlen(G,G/H) \le \mixlen(G,G/K).
  \]
\end{prop}
\begin{proof}
  Take $g_1,\dots,g_k\in G$ and Bernoulli random variables $\eps_1,\dots,\eps_k$ such that $\boldsymbol{g}=g_1^{\eps_1}\cdots g_k^{\eps_k}$ mixes the action of $G$ on~$G/K$. In particular, the random coset $\boldsymbol{g}K$ is uniform on the set $G/K$, and thus $\boldsymbol{g}H$ is uniform on the set $G/H$. This shows that the action of $G$ on $G/H$ is mixable.
\end{proof}

Restating Corollaries~\ref{cor:quotient} and~\ref{cor:N-and-G/N}, we proved:
\begin{thm}\label{thevar}
The class of mixable groups is closed under projection and extension.
\end{thm}

We can use \Cref{cor:quotient} and \Cref{cor:N-and-G/N} to establish the mixability of group extensions:

\begin{cor}\label{cor:products}
  Let $G_1$ and $G_2$ be finite groups.
  \begin{enumerate}
    \item $G_1\times G_2$ is mixable if and only if $G_1$ and $G_2$ are mixable.
    \item If $G_1$ and $G_2$ are mixable, then $G_1\rtimes G_2$ is mixable.
    \item If $G_1$ and $G_2$ are mixable, then $G_1\wr G_2$ is mixable.
  \end{enumerate}
\end{cor}

However, mixability does not pass to subdirect products:

\begin{exmpl}[Nonmixable subdirect product of mixable groups]
    For a finite group $G$ and two subgroups $H_1,H_2\le G$, if the actions of~$G$ on $G/H_1$ and on $G/H_2$ are mixable, it does not necessarily follow that the action of $G$ on $G/(H_1\cap H_2)$ is mixable.

    To demonstrate this, we take $G=\Alt{4}$ and the two subgroups $H_1=\sg{(1\,2\,3)}$ and $H_2=\sg{(2\,3\,4)}$. We will show that the actions of $G$ on $G/H_1$ and on $G/H_2$ are mixable. Indeed, $H_1$ and $H_2$ are the stabilizers of $4$ and $1$ respectively in the action of $\Alt{4}$ on $\set{1,2,3,4}$, and this action is mixable (we show this later in \Pref{prop:mix-An-act}, although this can be checked directly). However, $H_1\cap H_2=\set{\id}$, and $\Alt{4}$ is not mixable (see the paragraph before \Pref{prop:mix-An-act}).
\end{exmpl}

\section{Basic examples}\label{sec:basic-examples}

In this section we consider some basic examples where one can prove or disprove mixability with our current techniques. We begin by showing that $2$-groups are mixable, generalizing \Eref{exmpl:Z-mod-pow2}:

\begin{cor}\label{cor:2-groups}
    Every finite $2$-group $G$ is mixable, and $\mixlen(G)=\log_2\left|G\right|$.
\end{cor}

\begin{proof}
    If $G$ is a group of order $2^t$, then one can write
    \[
        \set{e} =G_t \normali G_{t-1} \normali\cdots\normali G_0=G,
    \]
    where ${[G_i:G_{i+1}]}=2$ for any $0\le i\le t-1$. Since $G_i/G_{i+1}\cong\Z/2\Z$ is mixable for all $i$, it follows that $G$ is mixable with $\mixlen(G)\le t$. The lower bound $\mixlen(G)\ge t$ follows from \Rref{rem:entropy-bound}.
\end{proof}

In contrast to $2$-groups, we show that groups of odd order are not mixable.

\begin{thm}\label{thm:odd-order}
    Any finite group with a nontrivial quotient of odd order is not mixable.
\end{thm}

\begin{proof}
    By \Cref{cor:N-and-G/N}, it is enough to prove the theorem for any nontrivial group $G$ of odd order. Suppose to the contrary that there exist $g_1,\dots,g_k\in G$ and independent Bernoulli random variables $\eps_1,\dots,\eps_k$ such that $g_1^{\eps_1}\cdots g_k^{\eps_k}$ is uniform on $G$. We may assume that $k$ is the mixing length of $G$, and thus the distribution $\mu$ of $\boldsymbol{g}=g_1^{\eps_1}\cdots g_{k-1}^{\eps_{k-1}}$ is not uniform on $G$. In particular, there exists some $a\in G$ such that $\mu(a)>\frac{1}{\left|G\right|}$.

    Write $p=\Pr(\eps_k=1)$. Since $\boldsymbol{g}g_k^{\eps_k}$ is uniform on $G$, for all $g\in G$ we have
    \begin{equation}\label{eq:odd-order-eq}
      \frac{1}{\left|G\right|}=(1-p)\mu(g)+p\mu(gg_k^{-1}).
    \end{equation}
    Therefore $\mu(g)>\frac{1}{\left|G\right|}$ if and only if $\mu(gg_k^{-1})<\frac{1}{\left|G\right|}$.

    Since $\mu(a)>\frac{1}{\left|G\right|}$, we have $\mu(ag_k^{-1})<\frac{1}{\left|G\right|}$, which in turn implies $\mu(ag_k^{-2})>\frac{1}{\left|G\right|}$, and so forth. Continuing this process yields $\mu(ag_k^m)>\frac{1}{\left|G\right|}$ for even $m$ and $\mu(ag_k^m)<\frac{1}{\left|G\right|}$ for odd $m$. Taking $m=\left|G\right|$, which is odd, proves that $\mu(a)=\mu(ag_k^{\left|G\right|})<\frac{1}{\left|G\right|}$, in contradiction to our choice of $a$. This concludes the proof.
\end{proof}

We will later see an alternative proof of the theorem using representation theory (see \Cref{quickodd} for a special case, and \Cref{cor:reps--1-ev} and the discussion afterward for arbitrary odd order groups).
\medskip

As a direct corollary of the two results in this section, we can determine which finite abelian groups are mixable:

\begin{cor}\label{cor:abelian}
    A finite abelian group is mixable if and only if it is a $2$-group.
\end{cor}

\begin{proof}
    Any abelian $2$-group is mixable by \Cref{cor:2-groups}, and the rest are not mixable since they have a nontrivial quotient of odd order.
\end{proof}

\section{The symmetric and alternating groups}\label{sec:perm-groups}

The next two cases we turn to are the symmetric groups $\Sym{n}$ and the alternating groups $\Alt{n}$. As mentioned above, the work of Angel and Holroyd shows that the symmetric groups~$\Sym{n}$ are mixable for all $n$. In this section, we show that one can mix~$\Sym{n}$ in at most $\frac{3}{2}\log_2(n!)+\frac{1}{2}n$ steps using arbitrary involutions, which agrees with our lower bound of \Rref{rem:entropy-bound} up to a factor of $2$. We then consider the alternating groups $\Alt{n}$, showing that they are mixable for all $n\ge 5$ in at most $\frac{3}{2}\log_2(n!)+\frac{1}{2}n$ steps, which is again of the same order as the lower bound of \Rref{rem:entropy-bound}.

\subsection{The symmetric groups}

To give an upper bound on the mixing length of $\Sym{n}$, we first show that its natural action on $\set{1,\dots,n}$ can be mixed in at most $2\log_2 n$ steps, and then use an inductive argument to conclude the mixing length bound. We begin with the case where $n$ is a power of $2$:

\begin{lem}\label{lem:sym-power-2-mix}
  The action of $\Sym{2^m}$ on $\set{1,\dots,2^m}$ can be mixed using involutions in $m$ steps.
\end{lem}

\begin{proof}
  We think of $\set{1,\dots,2^m}$ as sequences of bits of length $m$, and mix them ``bit by bit''. Formally, we consider the following mixing sequence:
  \begin{itemize}
    \item $\relem{(1,2)(3,4)\cdots(2^m-1,2^m)}{\frac{1}{2}}$;
    \item $\relem{(1,3)(2,4)\cdots(2^m-2,2^m)}{\frac{1}{2}}$;
    \item $\cdots$
    \item $\relem{(1,2^{m-1}+1)(2,2^{m-1}+2)\cdots(2^{m-1},2^m)}{\frac{1}{2}}$.
  \end{itemize}
  If the initial value we start with is $1$, one easily sees that its distribution after $j$ steps is uniform on $\set{1,\dots,2^j}$ for any $1\le j\le m$, which proves the lemma.
\end{proof}

We can now extend the lemma to arbitrary $n$. Let $\wt(n)$ denote the Hamming weight of $n$ (the number of $1$'s in the binary representation). Note that  $\wt(n) \leq \log_2(n)$.

\begin{prop}\label{prop:mix-Sn-act}
  The action of $\Sym{n}$ on $\set{1,\dots,n}$ can be mixed using arbitrary involutions in $\floor{\log_2 n} + \omega(n)-1$ steps.
\end{prop}

\begin{proof}
  The idea is to partition $\set{1,\dots,n}$ to blocks such that their sizes are powers of $2$. To mix the action of $\Sym{n}$ on $\set{1,\dots,n}$ starting from the element $1$, we first send it to a random block using transpositions, and then apply the previous lemma to mix each block separately.

  Write $n=\sum_{m\in I}2^m$ for $I\sub\set{1,\dots,\floor{\log_2 n}}$. Arrange the elements of $I$ in an ascending order $m_1<\cdots<m_t=\floor{\log_2 n}$, so that $t = \card{I} = \wt(n)$. Write $n_k=\sum_{j=1}^k 2^{m_j}$ for each $1\le k\le t$ (so that $n_t=n$), and define $B_k=\set{n_{k-1}+1,\dots,n_k}$ for all $1\le k\le t$.

  We define a random subproduct of permutations of $\Sym{n}$ that will send $1$ to a uniform element of $\set{1,\dots,n}$. We will do this in two stages.\medskip

  In stage I, we send $1$ to a random block among $B_1,\dots,B_t$. We do so using the following steps:
  \begin{itemize}
    \item Apply $(1,n_1+1)$ with probability $p_2=\frac{\left|B_2\right|}{n}$.
    \item Apply $(1,n_2+1)$ with probability $p_3=\frac{\left|B_3\right|}{n-\left|B_2\right|}$.
    \item $\cdots$
    \item Apply $(1,n_{t-1}+1)$ with probability $p_t=\frac{\left|B_t\right|}{n-(\left|B_2\right|+\cdots+\left|B_{t-1}\right|)}$.
  \end{itemize}
  Note that after we applied one transposition, the rest do not move the image of $1$.

  Let $\boldsymbol{\sigma}$ denote the random permutation achieved by these steps. Then, for any $2\le k\le t$,
  \[
    \Pr(\boldsymbol{\sigma}(1)\in B_k) = \prod_{j=2}^{k-1} (1-p_j) \cdot p_k =\frac{\left|B_k\right|}{n},
  \]
  which also implies that $\Pr(\boldsymbol{\sigma}(1)\in B_1)=\frac{\left|B_1\right|}{n}$.\medskip

  In stage II, we consider the subgroup $\Sym{2^{m_1}}\times\cdots\times\Sym{2^{m_t}}\le\Sym{n}$ which permutes each block separately. By \Lref{lem:sym-power-2-mix}, the $k$-th block $B_k$ can be mixed in at most $m_k$ steps for all $1\le k\le t$. Since we only care about mixing the blocks separately, and all probabilities involved are~$\frac{1}{2}$, we can apply the different permutations of \Lref{lem:sym-power-2-mix} on all blocks simultaneously. As the different blocks have different mixing lengths, we can extend the shorter mixings by the identity element (though, in fact, any extension would work).

  Let $\boldsymbol{\tau}$ denote the random permutation achieved by the steps of stage~II. We claim that $\boldsymbol{\tau\sigma}(1)$ is uniform on $\set{1,\dots,n}$. Indeed, fix $1\le i\le n$, and let $k$ be the index such that $i\in B_k$. Then
  \[
    \Pr(\boldsymbol{\tau\sigma}(1)=i) = \Pr(\boldsymbol{\tau\sigma}(1)=i\suchthat \boldsymbol{\sigma}(1)\in B_k)\Pr(\boldsymbol{\sigma}(1)\in B_k) = \frac{1}{\left|B_k\right|}\frac{\left|B_k\right|}{n}=\frac{1}{n}.
  \]
  Counting the mixing steps, stage I required $\wt(n)-1$ steps, while stage~II required $m_t=\floor{\log_2 n}$ steps. This concludes the proof.
\end{proof}

Summing up, we obtain an upper bound for the mixing length of $\Sym{n}$:

\begin{thm}\label{cor:mixlen-sym}
  Using involutions, one can achieve
  \[
    \mixlen(\Sym{n})\le \frac{3}{2}\floor{\log_2(n!)} + \frac{1}{2}n.
  \]
\end{thm}

\begin{proof}
  By \Pref{prop:mix-Sn-act}, the action of $\Sym{n}$ on $\set{1,\dots,n}$ can be mixed within $\floor{\log_2 n} + \wt(n)-1$ steps. Write $m = \floor{\log_2(n)}$, so that $2^m \leq n$. Writing the numbers $1,\dots,n$ in their binary representation, each of the $m+1$ bits involved is equal to $1$ at most $(n+1)/2$ times, so $\sum_{k=1}^n \wt(k) \leq \frac{1}{2}(m+1)(n+1)$.
  The point stabilizer of this action is $\Sym{n-1}$, hence applying \Pref{prop:H-and-G/H} repeatedly we have
\begin{align*}
\mixlen(\Sym{n}) & \leq  \sum_{k=1}^{n} ( \floor{\log_2 k} + \omega(k) - 1) \\
& =  \log_2(n!) + \sum_{k=1}^{n} (\omega(k)-1) \\
& \leq  \log_2(n!) + \frac{1}{2}(m+1)(n+1)-n \\
& =  \log_2(n!) + \frac{1}{2}(m-1)(n+1) + 1 \\
& =  \log_2(n!) + \frac{1}{2}(\floor{\log_2(n)}-1)(n+1) + 1 \\
& \leq  \frac{3}{2}\log_2(n!) + \frac{1}{2}n,
\end{align*}
as required, where the final stage follows from $2 n^{n+1}   \leq 4^nn!$ implying
$$(\log_2(n)-1)(n+1) + 2 \leq \log_2(n!) + n.$$
\end{proof}

\subsection{The alternating groups}

We next consider the mixability of the alternating groups $\Alt{n}$. For $n=3$, the group $\Alt{3}$ is cyclic of order $3$, hence is not mixable by \Tref{thm:odd-order}. When $n=4$, the group $\Alt{4}$ is again not mixable, since it has a quotient isomorphic to $\Alt{3}$. We will prove that $\Alt{n}$ is mixable whenever $n\ge 5$, with a similar strategy to the above.

\begin{prop}\label{prop:mix-An-act}
  For any $n\ge 4$, the action of $\Alt{n}$ on $\set{1,\dots,n}$ is mixable within at most $\floor{\log_2 n}+\omega(n)$ steps.
\end{prop}

\begin{proof}
  The idea is to adjust the mixing presented in the proof of \Pref{prop:mix-Sn-act} to use only even permutations. There are two places where we use odd permutations. In stage I, we can replace the transposition $(1,n_k+1)$ with $(1,n_k+1)(n-1,n)$. This does not affect the image of $1$, since by the assumption $n\ge 5$, the last block has size $\ge 4$. Therefore the outcome of stage I will remain valid.

  In stage II, all permutations used are even other than when mixing a block of size $2$, where the above proof uses a single transposition. We fix that by first mixing all other blocks in $\floor{\log_2 n}+\omega(n)-1$ steps, and then mixing the block of size $2$ by multiplying the required transposition with the transposition $(n-1,n)$. Note that this does not affect the last block, as it is already mixed.

  Our adjustments added at most $1$ new step to the above steps, hence we proved the proposition.
\end{proof}

\begin{cor}\label{cor:alternating}
  The alternating group $\Alt{n}$ is mixable for all $n\ge 5$, and
  \[
    \log_2(n!)-1\le\mixlen(\Alt{n})\le \frac{3}{2}\floor{\log_2 (n!)} + \frac{1}{2} n.
  \]
\end{cor}

\begin{proof}
  Consider the action of $\Alt{n}$ on $X=\set{\set{i,j}\suchthat 1\le i<j\le n}$ by applying the permutation pointwise, which is transitive when $n\ge 4$. We first claim that this action is mixable for all $n\ge 5$, in which case it can be mixed within at most $\floor{\log_2(n-1)}+\floor{\log_2 n}+\omega(n-1)+\omega(n)$ steps. Indeed, we can mix the first coordinate by the action of $\Alt{n}$ on $\set{1,\dots,n}$, and the second one by the action of the stabilizer $\Alt{n-1}$, with the above number of steps.

  Next, the stabilizer of the tuple $(n-1,n)$ under the above action is $\Alt{n-2}\le \Alt{n}$, which is contained in a copy of $\Sym{n-2}\le \Alt{n}$. Since the action of $\Alt{n}$ on $X$ is mixable, and $\Sym{n-2}$ is mixable, we may apply \Pref{prop:H-and-G/H} and \Pref{prop:G/K-and-G/H} and deduce that $\Alt{n}$ is mixable and
  \begin{align*}
    \mixlen(\Alt{n}) & \le \mixlen(\Alt{n},X) + \mixlen(\Sym{n-2}) \\
    & \le \floor{\log_2(n-1)}+\floor{\log_2 n}+\omega(n-1)+\omega(n) \\
    & \ \ \ + \sum_{k=1}^{n-2}(\floor{\log_2 k}+\omega(k)-1) \\
    & \le \frac{3}{2}\floor{\log_2 (n!)} + \frac{1}{2}n.
  \end{align*}
\end{proof}

\section{Mixability and representations}\label{sec:rep-theory}

mixability can be interpreted in terms of group representations. We begin with brief preliminaries from representation theory, and then relate it to mixability.

\subsection{Preliminaries}

Recall that any probability measure $\mu$ on $G$ can be viewed as an element $\widetilde{\mu}=\sum_{g\in G}\mu(g)g$ in the group algebra $\R[G]$, with positive coefficients summing up to $1$. For instance, the uniform measure $U_G$ on $G$ corresponds to the central idempotent $\widetilde{U_G}=\frac{1}{\left|G\right|}\sum_{g\in G}g$. Recall further that the convolution of two probability measures $\mu,\nu$ on $G$, given by
\[
    \mu*\nu(x):=\sum_{y\in G}\mu(xy^{-1})\nu(y)
\]
for any $x\in G$, corresponds under this identification to multiplication in the group algebra, i.e.\ $\widetilde{\mu*\nu}=\widetilde{\mu}\cdot\widetilde{\nu}$, which is again a probability measure.

Passing to the group algebra over $\C$, the \textbf{Fourier transform} of an element $f=\sum_{g\in G}\alpha_g g\in\C[G]$ at a representation $\rho\colon G\to\GL[d_{\rho}](\C)$ is the matrix given by
\[
    \widehat{f}(\rho) = \sum_{g\in G}\alpha_g\rho(g) \in \M[d_{\rho}](\C).
\]
Then one has $\widehat{f_1f_2}(\rho)=\widehat{f_1}(\rho)\widehat{f_2}(\rho)$ for all $f_1,f_2\in\C[G]$. Writing $\widehat{G}$ for the set of all irreducible representations of $G$, one can express the coefficients $\set{\alpha_g}$ using the inverse Fourier transform as
\[
    \alpha_g = \frac{1}{\left|G\right|}\sum_{\rho\in\widehat{G}} d_{\rho}\tr(\rho(g^{-1})\widehat{f}(\rho)).
\]

\subsection{Mixability via representations}

As mentioned in the Introduction, we can reformulate mixability as a special decomposition of the idempotent $\widetilde{U_G}=\frac{1}{\left|G\right|}\sum_{g\in G} g$ in the groups algebra $\C[G]$:

\begin{rem}[The Decomposition Criterion]
  A finite group $G$ is mixable if and only if there are elements $g_1,\dots,g_k\in G$ and probabilities $p_1,\dots,p_k\in[0,1]$ such that
  \[
    \prod_{i=1}^k ((1-p_i)e+p_ig_i) = \frac{1}{\left|G\right|}\sum_{g\in G}g.
  \]
\end{rem}

By Maschke's theorem, the group algebra $\C[G]$ decomposes as a direct sum of matrix algebras, corresponding to the irreducible representations of~$G$. It is therefore natural to look for a mixability condition in terms of the representations of $G$.

\begin{defn}
    Let $G$ be a finite group. We say that an irreducible representation $\rho\colon G\to\GL[d](\C)$ of $G$ is \textbf{mixable} if there are elements $g_1,\dots,g_k\in G$ and probabilities $p_1,\dots,p_k\in[0,1]$ such that
    \[
        \prod_{i=1}^k((1-p_i)I+p_i\rho(g_i))=0.
    \]
    In this case we call $\relem{g_1}{p_1},\dots,\relem{g_k}{p_k}$ a \textbf{mixing sequence} of $\rho$.
\end{defn}
More generally, we say that a representation of $G$ is {\bf{mixable}} if its nontrivial irreducible components are mixable. Note that
the trivial representation cannot be mixed because then $(1-p)I+p\rho(g)=1$ for any $g \in G$ and any~$p$.

\begin{thm}\label{thm:rep-mix}
    A finite group $G$ is mixable if and only if the regular representation of $G$ is mixable.
\end{thm}
Recall that the regular representation decomposes into a direct sum of all the irreducible representations (each $d$-dimensional representation repeated~$d$ times).
\begin{proof}
    If $G$ is mixable, then there exist elements $g_1,\dots,g_k\in G$ and probabilities $p_1,\dots,p_k\in[0,1]$ such that
    \[
        \prod_{i=1}^k ((1-p_i)e+p_ig) = \frac{1}{\left|G\right|}\sum_{g\in G}g.
    \]
    Let $\rho$ be a nontrivial irreducible representation of $G$. Applying $\rho$ on the above equation yields
    \[
        \prod_{i=1}^k ((1-p_i)I+p_i\rho(g)) = \frac{1}{\left|G\right|}\sum_{g\in G}\rho(g).
    \]
    Since $\rho$ is nontrivial and irreducible, it follows that $\sum_{g\in G}\rho(g)=0$, proving that $\rho$ is mixable.

    Conversely, suppose that all nontrivial irreducible representations of~$G$ are mixable. Concatenating all sequences of elements from each nontrivial irreducible representation, we conclude that there exist elements $g_1,\dots,g_k\in G$ and probabilities $p_1,\dots,p_k\in[0,1]$ such that
    \begin{equation}\label{eq:mix-reps-f}
        \prod_{i=1}^k ((1-p_i)I+p_i\rho(g)) = 0
    \end{equation}
    for any nontrivial irreducible representation $\rho$ of $G$.

    Write $f=\prod_{i=1}^k ((1-p_i)e+p_ig)\in\C[G]$; we will prove $f=\frac{1}{\left|G\right|}\sum_{g\in G}g$, which shows that $G$ is mixable. Indeed, for any nontrivial irreducible representation $\rho$ of $G$ we have $\widehat{f}(\rho)=0$ by \eqref{eq:mix-reps-f}. The inverse Fourier transform formula then shows that $f=\sum_{g\in G}\alpha_g g$, where
    \[
        \alpha_g = \frac{1}{\left|G\right|}\sum_{\rho\in\widehat{G}} d_{\rho}\tr(\rho(g^{-1})\widehat{f}(\rho)) = \frac{1}{\left|G\right|}\tr(\One(g^{-1})\widehat{f}(\One)) = \frac{1}{\left|G\right|}
    \]
    ($\One$ denotes the trivial representation of $G$). This concludes the proof.
\end{proof}

For example, we immediately get the following special case of \Tref{thm:odd-order} (which implies it if one cares to apply Feit-Thompson, since every odd-order group has a cyclic quotient).
\begin{cor}\label{quickodd}
A cyclic group of odd order $n$ cannot be mixable.
\end{cor}
\begin{proof}
We show that no nontrivial irreducible representation~$\rho$ of the cyclic group $\sg{g | g^n = 1}$ is mixable when $n$ is odd. Indeed $\rho(g^i)$ is not real for any $0 \neq i \in \Z/n\Z$, so $(1-p)+p\rho(g^i)$ is never zero.
\end{proof}

\subsection{The induced representation}

Let $G$ be a group with a mixable subgroup $H \leq G$. When $H$ is normal, \Tref{thevar} implies that~$G$ is mixable if and only if $G/H$ is mixable. But we can gain a lot from mixability of a nonnormal subgroup.

A representation of $G$ is automatically a representation of $H$, by restriction. On the other hand, if $\rho$ is a representation of of $H$, there is an induced representation $\rho|^G$, of dimension $\dimcol{G}{H}\dim \rho$, which may be viewed as a sending each $g \in G$ to a permutation block matrix. Let $\One_H$ denote the trivial representation of~$H$. Thus $(\One_H)|^G$ is a representation of dimension $\dimcol{G}{H}$ of $G$. In fact, the induced representation $(\One_H)^G$ is isomorphic to the action of $G$ on $G/H$, when the latter is viewed as a map sending each $g \in G$ to a permutation matrix indexed by $G/H$. This can be used to prove the next theorem by \Pref{prop:H-and-G/H}, but we prefer to also give a proof through representations.

\begin{thm}\label{thm:rep-mix++}
Let $G$ be a finite group with a mixable subgroup~$H$. Then $G$ is mixable if and only if $(\One_H)|^G$ is mixable.
\end{thm}
\begin{proof}
By the previous theorem, $G$ is mixable if and only if all the nontrivial irreducible representations of $G$ are mixable, so let~$\rho$ be a nontrivial irreducible representation over $G$.

The reduced representation $\rho|_H$ decomposes as a direct sum $\rho|_H = \psi_1 \oplus \cdots \oplus \psi_s$, where the $\psi_j$ are irreducible. Since $H$ is mixable, every nontrivial $\psi_j$ is mixable by $H$, and consequently by $G$. This shows that if $\One_H$ is not a subrepresentation of $\rho|_H$, then $\rho|_H$ is mixable over $H$, so~$\rho$ is mixable over $G$.

The Hermitian inner product in $\C[G]$ induces an inner product on the space of class functions, which is the space spanned by characters, and these are in one-to-one correspondence with representations. The same holds for $H$. The inner product $\innprod{\One_H}{\rho|_H}_H$ counts how many times $\One_H$ appears as a direct summand in the irreducible decomposition of $\rho|_H$.  Similarly, $\innprod{(\One_H)|^G}{\rho}_G$ counts how many times $\rho$ appears as a direct summand in $(\One_H)^G$.
By Frobenius' reciprocity theorem, the two numbers are equal. Summarizing, if $\rho$ is not contained in the induced representation $(\One_H)|^G$, then
$\innprod{\One_H}{\rho|_H}_H = \innprod{(\One_H)|^G}{\rho}_G = 0$, so~$\rho$ is taken care of by the mixability of $H$. The remaining irreducible representations are those contained in $(\One_H)^G$.
\end{proof}

The full potential of this theorem is when $H_1,\dots,H_k$ are the maximal known mixable subgroups of $G$. Then, by the proof of this theorem, mixability of $G$ follows once we verify that the nontrivial subrepresentations which are summands in each induced representation $(\One_{H_j})|^G$ are mixable. (It suffices to take one representative from each conjugacy class of subgroups).

A more direct application follows from the fact that a $2$-Sylow subgroup is always mixable by \Cref{cor:2-groups}.
\begin{cor}
Let $P$ be a $2$-Sylow subgroup of a finite group~$G$. Then~$G$ is mixable if and only if $(\One_P)|^G$ is mixable.
\end{cor}
This is a strict improvement over \Tref{thm:rep-mix}, since $\dimcol{G}{P} < \card{G}$ for any even order group $G$ (the odd order groups are not mixable).

\subsection{Necessary conditions for mixability}

We turn to give a necessary condition for the mixability of a representation. The heart of the argument lies in the following lemma:

\begin{lem}\label{lem:reps-sing}
  Let $G$ be a finite group, let $g\in G$ be some element of~$G$, let $0\le p\le 1$, and let $\rho\colon G\to\GL[d](\C)$ be a representation of $G$.

  Then the matrix $(1-p)I+p\,\rho(g)$ is singular if and only if $-1$ is an eigenvalue of $\rho(g)$ and $p=\frac{1}{2}$. In this case, the order of $g$ must be even.
\end{lem}

\begin{proof}
  We first note that $(1-p)I+p\,\rho(g)$ is singular if and only if $-\frac{1-p}{p}$ is an eigenvalue of $\rho(g)$. Recall that if $g$ has order $n$ in $G$, then $\rho(g)^n=I$, hence all eigenvalues of $\rho(g)$ are roots of unity of order dividing $n$. Since $-\frac{1-p}{p}$ is a real number, this is only possible if $-\frac{1-p}{p}=\pm 1$, which is only possible if $p=\frac{1}{2}$ and $-\frac{1-p}{p}=-1$. Furthermore, in this case $n$ must be even, since $(-1)^n=1$.
\end{proof}

Following this lemma, we say that an irreducible representation $\rho$ of $G$ is {\bf{potentially mixable}} if there is $a \in G$ such that $\rho(a)$ has eigenvalue $-1$. We show that this is a necessary condition:

\begin{cor}\label{cor:reps--1-ev}
  Let $G$ be a finite group. Then any mixable irreducible representation $\rho\colon G\to\GL[d](\C)$ is potentially mixable.

  Furthermore, let $\relem{g_1}{p_1},\dots,\relem{g_k}{p_k}$ be a mixing sequence of $\rho$, of minimal length. Then $p_1 = p_k = \frac{1}{2}$, and the order of~$g_1$ and~$g_k$ is even. In fact,~$g_1$ and~$g_k$ can be chosen to have order a power of $2$.
\end{cor}
\begin{proof}
  Assume that $\relem{g_1}{p_1},\dots,\relem{g_k}{p_k}$ is a mixing sequence of $\rho$, of minimal length. Denote $a_i = (1-p_i)I+p_i \rho(g_i) \in \M[d](\C)$ where $d = \chi_{\rho}(1)$ is the dimension of the representation. Then $a_1 \cdots a_k = 0$ in the group algebra. By minimality, $a_1$ and $a_k$ are singular. Then \Lref{lem:reps-sing} shows that $p_1=p_k=\frac{1}{2}$, that $-1$ is an eigenvalue of both $\rho(g_1)$ and $\rho(g_k)$, and that the order of $g_1$ and $g_k$ in $G$ is even.

  It remains to show that $g_1$ can be replaced by some $g_1'\in G$ of order a power of $2$ (the claim for  $g_k$ is similar). Write the order of $g_1$ in $G$ as $2^tm$, where $m$ is odd, and let $g_1'=g_1^m$. Note that
  \[
    \frac{1}{2}I+\frac{1}{2}\rho(g_1') = \underbrace{\left(I-\rho(g_1)+\cdots+\rho(g_1^{m-1})\right)}_A\left(\frac{1}{2}I+\frac{1}{2}\rho(g_1)\right),
  \]
  and thus by $\left(\frac{1}{2}I+\frac{1}{2}\rho(g_1')\right)a_2 \cdots a_k = A a_1 \cdots a_k = 0$.
\end{proof}

While being potentially mixable seems only like a necessary criterion for an irreducible representation to be mixable, we conjecture that this is also sufficient; see \Sref{sec:struc} for a further discussion and partial results in this direction.

The representation viewpoint allows us to give alternative proofs for some of our previous results. For instance, \Tref{thm:odd-order}, stating that any nontrivial group $G$ of odd order cannot be mixable, follows immediately from the second part of the above corollary. Additionally, several results of \Sref{sec:basic-defs} can be proved by analyzing appropriate representations of groups.

\begin{rem}
  While the proof of \Cref{cor:reps--1-ev} only used the singularity of matrices of the form $(1-p)I+p\,\rho(g)$, one can extract further information from the dimension of their kernels. Indeed, if
  \[
    \prod_{i=1}^k((1-p_i)I+p_i\rho(g_i))=0,
  \]
  then
  \[
    \sum_{i=1}^k\dim_C\ker((1-p_i)I+p_i\rho(g_i))\ge d.
  \]
  Recalling that $(1-p_i)I+p_i\rho(g_i)$ is non-singular unless $p_i=\frac{1}{2}$, we get
  \[
    \sum_{i:p_i=\frac{1}{2}}\dim_\C\ker(I+\rho(g_i))\ge d,
  \]
  which can be used to get a lower bound on the number of indices $i$ for which $p_i=\frac{1}{2}$. This provides an alternative proof of \cite[Theorem 1.4]{AH}.
\end{rem}

\section{Structure and mixability}\label{sec:struc}

\Tref{thm:odd-order} is the first hint that mixability is related to classical properties of the group. In particular, it directs our attention to the maximal odd-order quotient. For any finite group $G$, let $U(G)$ denote the normal closure of some (and thus every) $2$-Sylow subgroup, which is the subgroup generated by all the $2$-Sylow subgroups. By definition~$U(G)$ is a normal subgroup of odd index. It can be viewed as a counterpart to the maximal normal subgroup of odd order, often denoted $O(G)$ \cite{Segev}.

\begin{prop}\label{maxodd}
Let $G$ be a finite group. Then $U(G)$ is the minimal normal subgroup of odd index.
\end{prop}
\begin{proof}
Let $N$ be any normal subgroup of $G$. Then $\dimcol{G}{N}$ is odd if and only if the same power of $2$ divides both $\card{G}$ and $\card{N}$; if and only if a $2$-Sylow subgroup of $N$ is a $2$-Sylow subgroup of $G$; if and only if~$N$ contains $U(G)$ because of normality.
\end{proof}
Consequently, $G/U(G)$ is the maximal quotient of $G$ of odd order.
Let us say that $G$ is {\bf{$2'$-simple}} if $G$ has no odd-index normal subgroups,
namely, $U(G) = G$. In particular, simple groups are $2'$-simple. By \Pref{maxodd}, $G$ is $2'$-simple if and only if it is generated by its $2$-Sylow subgroups. A subgroup whose normal closure is the full group is called ``contranormal''. Thus, a group is $2'$-simple if and only if its $2$-Sylow subgroups are contranormal.

By \Tref{thm:odd-order}, if $G$ is mixable then $G/U(G)$ must be trivial. Thus:
\begin{cor}
If $G$ is mixable then it is $2'$-simple.
\end{cor}

Considering all of our examples, it seems as if  the only obstruction to mixability is the existence of a nontrivial quotient of odd order. In fact, all the examples discussed in this paper can be mixed by using only elements of order $2^t$ for some $t\ge 1$. We therefore pose the following conjecture:
\begin{conj}\label{conj:mix-quotients}
A finite group is mixable if and only if it is $2'$-simple.
\end{conj}

To gain some support of this conjecture, let us now utilize \Lref{lem:reps-sing}.

\begin{prop}\label{prop:equiv-conj}
A finite group $G$ is $2'$-simple if and only if any nontrivial representation $\rho\colon G\to\GL[d](\C)$ is potentially mixable.
\end{prop}

\begin{proof}
($\Longrightarrow$) Assume $G$ is $2'$-simple. Let $\rho\colon G\to\GL[d](\C)$ be a nontrivial representation of~$G$. By semisimplicity of $\C[G]$, and since eigenvalues accumulate, we may assume $\rho$ is irreducible. Since $\rho$ is nontrivial, it cannot vanish on all $2$-Sylow subgroups of $G$, as they generate $G$, so there must be some $g\in G$ of order $2^t$ for some $t\ge 1$ such that~$\rho(g)\ne I$. Take the maximal $0\le j\le t-1$ such that $\rho(g^{2^j})\ne I$; then $a= \rho(g^{2^j})$ is a matrix satisfying $a^2=I$ but $a\ne I$, so $-1$ must be an eigenvalue.

($\Longleftarrow$) Assume, on the contrary, that $G$ has an odd-index normal subgroup $N$. Lift any nontrivial irreducible representation of $G/N$ to an irreducible representation $\rho$ of $G$. For any $a\in G$, the eigenvalues of $\rho(a)$ are roots of unity of order dividing the order of~$a$ in $G/N$, which is odd since it divides $\dimcol{G}{N}$. Therefore, $-1$ cannot be an eigenvalue of $\rho(a)$, contradicting the assumption.
\end{proof}

We can thus restate \Cjref{conj:mix-quotients}:
\begin{conj}\label{conj:mix-representations}
If every nontrivial irreducible representation of $G$ is potentially mixable, then $G$ is mixable.
\end{conj}

The proof of \Pref{prop:equiv-conj} shows that if $-1$ is an eigenvalue of some element, then it is an eigenvalue of some element of order $2$ in the representation. Let us say that a group $H$ is {\bf{\twogen}} if it is generated by elements of order $2$.
(We use this terminology to avoid the collision with the term $2$-generated for groups generated by two elements).
Some remarks on such groups can be found in \cite{2gen}, where the authors note that there are plenty of \twogen\ groups. Indeed, every quotient of a Coxeter group is \twogen\ in our sense, as well as every finite simple group. As we saw above, a \twogen\ group is generated by the $2$-Sylow subgroups, and is thus $2'$-simple.

We put forth a conjecture that clearly follows from \Cjref{conj:mix-quotients}.
\begin{conj}\label{newconj}
Every finite \twogen\ group is mixable.
\end{conj}

Note that the class of \twogen\ groups is closed under taking quotients, but not under passing to subgroups (\eg $A_3 \normali S_3$).
The maximal \twogen\ subgroup of $G$, namely, the subgroup generated by all the elements of order $2$ in $G$, is normal. This leads to a useful normal series. (We say that a chain of subgroups of $G$ is {\bf{subnormal}} if each entry is normal in the next, and {\bf{normal}} if the entries are normal in~$G$).
\begin{thm}\label{Iup}
Every finite group $G$ has a normal series
$$1 = I_0 \normali I_1 \normali I_2 \normali \dots \normali I_t \normal G,$$
where the top quotient $G/I_t$ has odd order, and all the other quotients $I_{j+1}/I_j$ are \twogen.
\end{thm}
This is another indication that the class of \twogen\ groups is quite large:  by the Feit-Thompson Theorem, all the nonabelian simple composition factors of an arbitrary group $G$ hide in the \twogen\ quotients of the normal series.
\begin{proof}
We prove by induction on $\nu_2(\card{G})$, the maximal $s$ such that $2^s \mid \card{G}$. If $\card{G}$ is odd we are done by $1 \normali G$. Otherwise, let $I(G)$ denote the subgroup generated by the elements of order $2$, which is normal and has even order. Thus $\nu_2(G/I(G)) < \nu_2(G)$, and we are done by lifting a normal series of $G/I(G)$ to a normal series of~$G$ by attaching $I(G)$ as the first nontrivial subgroup in the new series.
\end{proof}

Note that in this series, we may take $I_t = U(G)$. In particular every $2'$-simple group has a normal series with \twogen\ quotients.
The opposite is true as well.
\begin{thm}\label{2'=2+2+2}
A finite group $G$ is $2'$-simple if and only if it has a subnormal series with \twogen\ quotients.
\end{thm}
\begin{proof}
Let $1 = I_0 \normali I_1 \normali \cdots \normali I_t = G$ be a subnormal series with \twogen\ quotients. Let $N \normal G$ be a normal subgroup of odd order. We prove that $N = G$. If $t= 0$ there is nothing to show. Assume $t \geq 1$.
Note that $I_1/(N \cap I_1) \cong NI_1 / N$ is \twogen\ as a quotient of~$I_1$, and of odd order as a subgroup of $G/N$, so it must be trivial. It follows that $I_1 \leq N$. Divide by $I_1$ and conclude by induction on $t$.
\end{proof}

We now see that it suffices to verify \Cjref{conj:mix-quotients} for \twogen\ groups.

\begin{prop}
Conjectures \ref{conj:mix-quotients} and \ref{newconj} are equivalent.
\end{prop}
\begin{proof}
That \Cjref{conj:mix-quotients} implies \Cjref{newconj} is obvious, because every \twogen\ group is $2'$-simple. The other direction is a consequence of the normal series obtained above by \Tref{thevar}.
\end{proof}

\section{Low-dimensional representations}\label{sec:8}

This section is devoted to proving \Cjref{conj:mix-representations} for low-dimensional representations.

\begin{rem}\label{rem:mix-rep-d}
  Let $\rho\colon G\to\GL[d](\C)$ be a representation of $G$. If there exists some $a\in G$ such that $\rho(a)=-I$, then $\frac{1}{2}I+\frac{1}{2}\rho(a)=0$, showing that $\rho$ is mixable.

  In particular, if $\rho\colon G\to\C^{\times}$ is a nontrivial irreducible $1$-dimensional representation of $G$, then $\rho$ is mixable if and only if there exists $a\in G$ such that $\rho(a)=-1$.
\end{rem}

We turn to proving the above mentioned conjecture for complex $2$-dimensional and real $3$-dimensional irreducible representations. Recall that any representation of a finite group is equivalent to a unitary one; hence we may add this assumption whenever needed.

\begin{prop}\label{prop:mix-rep-d-1}
    Let $\rho\colon G\to\GL[d](\C)$ be an irreducible representation of $G$ for some $d\ge 2$. If there exists $a\in G$ such that the multiplicity of $-1$ as an eigenvalue of $\rho(a)$ is $d-1$, then $\rho$ is mixable in at most $3$ steps.
\end{prop}

\begin{proof}
    As mentioned above, we may assume that $\rho$ is unitary, and that all eigenvalues of $\rho(a)$ are $\pm 1$. Hence there is an orthonormal basis $v_1,\dots,v_d$ of $\R^d$ such that $\rho(a)v_1=v_1$ and $\rho(a)v_i=-v_i$ for all $2\le i\le d$. We assume that $v_1,\dots,v_d$ are the standard basis of $\C^d$, so that $\rho(a)=\operatorname{diag}(1,-1,-1,\dots,-1)$.

    Recall that $\sum_{g\in G}gag^{-1}\in\Zent(\C[G])$. Therefore, by Schur's lemma, $\sum_{g\in G}\rho(gag^{-1})=\lambda I$ for some $\lambda\in\C$. Note that $\tr\rho(gag^{-1})=\tr\rho(a)=2-d$ for all $g\in G$, hence
    \[
        (2-d)\left|G\right|=\tr\left(\sum_{g\in G}\rho(gag^{-1})\right)=\tr(\lambda I)=\lambda d.
    \]
    In particular, $\lambda$ is a non-positive real number.

    By assumption, $\rho(g)$ is unitary for all $g\in G$. Therefore, the element in position $(1,1)$ in $\rho(gag^{-1})$ is
    \begin{align*}
        \rho(gag^{-1})_{1,1} &= (\rho(g)\rho(a)\rho(g)^{-1})_{1,1} = (\rho(g)\rho(a)\rho(g)^*)_{1,1} \\
        & = \sum_{i,j=1}^d \rho(g)_{1,i}\rho(a)_{i,j}(\rho(g)^*)_{j,1} = \sum_{i=1}^d \rho(g)_{1,i}\rho(a)_{i,i}\overline{\rho(g)_{1,i}} \\
        & = \left|\rho(g)_{1,1}\right|^2-\sum_{i=2}^d\left|\rho(g)_{1,i}\right|^2.
    \end{align*}
    Therefore, $\rho(gag^{-1})_{1,1}\in\R$ for all $g\in G$. Since $\sum_{g\in G}\rho(gag^{-1})_{1,1}=\lambda\le 0$ and $\rho(a)_{1,1}=1$, it follows that $\alpha\coloneqq\rho(hah^{-1})_{1,1}<0$ for some $h\in G$.

    We construct an explicit mixing sequence of $\rho$ as follows. Apply first $A=\frac{1}{2}I+\frac{1}{2}\rho(a)$, which annihilates $v_2,\dots,v_d$ and fixes $v_1$. Hence, $A$ maps $\R^d$ to $\R v_1$. Next, write $p=\frac{1}{1-\alpha}$, and apply the matrix $B=(1-p)I+p\,\rho(hah^{-1})$. As $B_{1,1}=1-p+p\,\rho(hah^{-1})_{1,1}=0$, it follows that $Bv_1\in \C v_2+\cdots+\C v_d$. Finally, apply $A$ again to annihilate the remaining subspace. This shows that~$\rho$ is mixable with a mixing sequence of length $3$.
\end{proof}

As a direct corollary, we prove that $2$-dimensional real representations satisfy \Cjref{conj:mix-representations}:

\begin{cor}\label{cor:mix-2-dim-real}
A potentially mixable irreducible $2$-dimensional representation is mixable, in at most $3$ steps.
\end{cor}

\begin{proof}
  The ``only if'' direction follows from \Cref{cor:reps--1-ev}. The ``if'' direction follows from a combination of \Rref{rem:mix-rep-d} and \Pref{prop:mix-rep-d-1}.
\end{proof}

This corollary allows us to establish \Cjref{conj:mix-quotients} for certain semidirect products:

\begin{prop}\label{prop:abelian-Z2}
    Let $G=A\rtimes(\Z/2\Z)$ be a semidirect product where $A$ is abelian. Then~$G$ is mixable if and only if it is $2'$-simple.
\end{prop}

\begin{proof}
    Since all irreducible representations of $A$ are $1$-dimensional, it follows that all irreducible representations of $G$ are $1$- or $2$-dimensional (see e.g.\ \cite[Section 8.2]{Serre}). Therefore, if $G$ is $2'$-simple, then all of its nontrivial irreducible representations are potentially mixable, and thus mixable by \Cref{cor:mix-2-dim-real}.
\end{proof}

We next consider $3$-dimensional real representations. We first show that we can send each vector to its orthogonal complement, generalizing the last step in the proof of \Pref{prop:mix-rep-d-1} for real representations:

\begin{lem}\label{lem:kill-vec}
    For any nontrivial irreducible representation $\rho\colon G\to\GL[d](\R)$, and any $0\ne v\in\R^n$, there exist $g\in G$ and $0\le p\le 1$ such that $(1-p)I+p\,\rho(g)$ sends $v$ to the orthogonal space $v^{\perp}$.
\end{lem}

\begin{proof}
    Since $\rho$ is nontrivial and irreducible, we have $\sum_{g\in G}\rho(g)=0$. In particular, $\sum_{g\in G}\rho(g)v=0$, so there is some $g\in G$ such that $\alpha\coloneqq\left\langle \rho(g)v,v\right\rangle\le 0$. Take $p=\frac{\left\|v\right\|}{\left\|v\right\|-\alpha}$, and note that $0\le p\le 1$ since $\alpha\le 0$. Then
    \begin{align*}
        \left\langle((1-p)I+p\,\rho(g))v,v\right\rangle & = \left\langle(1-p)v+p\,\rho(g)v,v\right\rangle \\
        & = (1-p)\left\|v\right\| + p\alpha=0
    \end{align*}
    as required.
\end{proof}

\begin{prop}\label{prop:mix-3-dim-real}
A potentially mixable irreducible $3$-dimensional real representation is mixable, in at most $7$ steps.
\end{prop}

\begin{proof}
As in the proof of \Pref{prop:mix-rep-d-1}, we may assume that the representation $\rho$ is unitary, and that for a suitable element $a \in G$ all eigenvalues of $\rho(a)$ are $\pm 1$. By replacing $a$ with an appropriate power, we may also assume that the order of $a$ is $2^t$ for some $t\ge 1$ (while $\rho(a)\ne I$).

  Write $U=\set{v\in\R^3\suchthat \rho(a)v=v}$. If $\dim U<2$, then the multiplicity of $-1$ as an eigenvalue of $\rho(a)$ is $2$ or $3$, and then $\rho$ is mixable either by \Rref{rem:mix-rep-d} (if $\dim U=0$) or by \Pref{prop:mix-rep-d-1} (if $\dim U=1$). We therefore assume that $\dim U=2$.

  Take $g\in G$ such that $gU\ne U$, which exists since $\rho$ is irreducible. Note that $U_0=U\cap gU$ is a $1$-dimensional subspace of $\R^3$. The subgroup $H=\sg{a,gag^{-1}}$ fixes $U_0$, hence acts on $U_0^{\perp}$ since $\rho$ is unitary. This action is irreducible, since the subspace fixed by $H$ is precisely $U_0$. Hence, by \Cref{cor:mix-2-dim-real}, there exists a random subproduct $\boldsymbol{h}$ in $H$ (of length at most $3$) that annihilates $U_0^{\perp}$. Note further that $\boldsymbol{h}$ fixes $U_0$, so $\boldsymbol{h}$ maps~$\R^3$ to $U_0$. Next, by \Lref{lem:kill-vec}, there exist $h\in G$ and $0\le p\le 1$ such that $(1-p)I+p\,\rho(h)$ maps $U_0$ to $U_0^{\perp}$ (since $\dim U=1$). Applying $\boldsymbol{h}$ again annihilates $U_0^{\perp}$, completing the mixing of $\rho$.

\end{proof}

\section{Finite Coxeter groups}\label{sec:coxeter}

Finite Coxeter groups, which are of course \twogen, pose a natural family of candidates for being mixable.
Recall that a \textbf{Coxeter group} is a group $G$ with a presentation of the form
\[
    G = \sg{r_1,\dots,r_n\suchthat (r_ir_j)^{m_{ij}}=1}
\]
where $m_{ii}=1$ and $m_{ij}=m_{ji}\ge 2$ is an integer or $\infty$ for all $1\le i,j\le n$. The finite Coxeter groups rise as finite Euclidean reflection groups, but this is no longer the case for the infinite ones. A Coxeter group is determined by the Coxeter-Dynkin diagram, whose vertices are the generators, with edges connecting $r_i,r_j$ when $m_{ij}>2$. The group is {\bf{irreducible}} if the diagram is connected.
When considering mixability, it suffices to assume the group is irreducible, as each Coxeter group is a direct product of irreducible Coxeter groups. We refer the reader to \cite{CoxeterMoser80,Humphreys90} for further details on Coxeter groups.

Using all the above tools, we are able to study the class of finite Coxeter groups. We focus on the irreducible Coxeter groups, since mixability is preserved under direct products.

\begin{thm}\label{thm:coxeter}
  All finite irreducible Coxeter groups, with the possible exception of those of types $\Cox{E}{6}$, $\Cox{E}{7}$ and $\Cox{E}{8}$, are mixable.
\end{thm}

We do not know whether the three remaining groups are mixable or not.

\begin{proof}
  We use the classification of finite irreducible Coxeter groups, proving that each one is mixable by the various techniques obtained above.

  \begin{itemize}
    \item $\Cox{A}{n}$. The Coxeter group of type $\Cox{A}{n}$ is isomorphic to $\Sym{n+1}$, which is mixable by \Eref{exmpl:Sn-mix}.
    \item $\Cox{B}{n}$. This Coxeter group is the group of signed permutations, which can be written as $(\Z/2\Z)\wr\Sym{n}$. This group is mixable by the mixability of $\Sym{n}$ and \Cref{cor:products}.
    \item $\Cox{D}{n}$. This is a subgroup of index $2$ in $\Cox{B}{n}$. This group is isomorphic to a semidirect product $(\Z/2\Z)^{n-1}\rtimes\Sym{n}$, and hence mixable by \Cref{cor:products}.
    \item $\Cox{F}{4}$. The Coxeter group $\Cox{F}{4}$ is isomorphic to a semidirect product $\Cox{F}{4}\cong\Cox{D}{4}\rtimes\Sym{3}$, hence it is mixable.
    \item $\Cox{H}{3}$. The coxeter group $\Cox{H}{3}$ is isomorphic to a direct product $\Alt{5}\times\Z/2\Z$, and is thus mixable.
    \item $\Cox{H}{4}$. The composition factors of $\Cox{H}{4}$ are $\Z/2\Z,\Z/2\Z,\Alt{5},\Alt{5}$, which are all mixable, and thus $\Cox{H}{4}$ is mixable as well by \Cref{cor:N-and-G/N}.
    \item $\Cox{I}{2}(n)$. The Coxeter group of type $\Cox{I}{2}(n)$ is the dihedral group $\Dih{n}$ of order $2n$. This group is a semidirect product $(\Z/n\Z)\rtimes(\Z/2\Z)$, and is therefore mixable by \Pref{prop:abelian-Z2}.
  \end{itemize}
\end{proof}

Following the proof of the theorem, one may get upper bounds on the mixing lengths of the finite Coxeter groups. We record the nontrivial one in the following corollary:

\begin{cor}\label{cor:dihedral}
    If $n=2^tm$ for odd $m$, then $\mixlen(\Dih{n})\le t+m$.
\end{cor}

\begin{proof}
  Write $\Dih{n} = \sg{\sigma,\tau\suchthat \sigma^n=\tau^2=(\tau\sigma)^2=1}$, and consider the normal subgroup $N=\sg{\sigma^m}\vartriangleleft\Dih{n}$. Since $N\cong\Z/2^t\Z$ is mixable and $\Dih{n}/N\cong\Dih{m}$ is mixable, we deduce that
  \[
    \mixlen(\Dih{n})\le\mixlen(\Z/2^t\Z)+\mixlen(\Dih{m})=t+\mixlen(\Dih{m}).
  \]
  It is therefore enough to show that $\mixlen(\Dih{m})\le m$.

  To prove this bound, we will find a sequence $\relem{g_1}{p_1}, \dots, \relem{g_m}{p_m}$ of length $m$ that mixes all nontrivial irreducible representations of $\Dih{m}$. Abusing notation, we denote by $\sigma,\tau$ the generators of $\Dih{m}$ as above. Recall that the nontrivial irreducible representations of $\Dih{m}$ are the following:
  \begin{itemize}
    \item One $1$-dimensional representation, induced from the natural projection $\Dih{m}\to \Dih{m}/\sg{\sigma}\cong\set{\pm 1}$;
    \item $k\coloneqq\frac{m-1}{2}$ $2$-dimensional representations $\rho_1,\dots,\rho_k$, where $\rho_j$ is given by
    \[
        \sigma\mapsto\begin{pmatrix}\cos\frac{2\pi j}{n}&\sin\frac{2\pi j}{n}\\-\sin\frac{2\pi j}{n}&\cos\frac{2\pi j}{n}\end{pmatrix},\;\;\;\tau\mapsto\begin{pmatrix}1&0\\0&-1\end{pmatrix}
    \]
    for each $1\le j\le k$.
  \end{itemize}
  By the proof of \Cref{cor:mix-2-dim-real}, for each $1\le j\le k$ there exists some $g_j\in \Dih{m}$ and $0\le p_j\le 1$ such that the sequence $\relem{\tau}{\frac{1}{2}},\relem{g_j}{p_j},\relem{\tau}{\frac{1}{2}}$ mixes the representation $\rho_j$. Therefore, the sequence
  \[
    \relem{\tau}{\frac{1}{2}},\relem{g_1}{p_1},\relem{\tau}{\frac{1}{2}},\relem{g_2}{p_2},\dots,\relem{g_k}{p_k},\relem{\tau}{\frac{1}{2}}
  \]
  mixes all nontrivial irreducible representations of $\Dih{m}$, hence mixes the group $\Dih{m}$. This sequence is of length $2k+1=m$, hence $\mixlen(\Dih{m})\le m$ as required.
\end{proof}

\section{\texorpdfstring{Mixability via $2$-transitive actions}{Mixability via 2-transitive actions}}\label{sec:2-transitive}

We have seen the use of group actions and representations to prove the mixability of groups. In this section, we will show how to use higher transitivity of a group action to prove the mixability of the acting group.\medskip

Recall that an action of a group $G$ on a set $X$ is called \textbf{$2$-transitive}, if for any $x_1\neq x_2\in X$ and $y_1\neq y_2\in X$ there exists $g\in G$ such that $gx_1=y_1$ and $gx_2=y_2$. In this case, for any $x\in X$ the (one-point) stabilizer $\stab_G(x)$ acts transitively on $X\setminus\set{x}$.

\begin{thm}\label{thm:2-transitive-act}
  Let $G$ be a finite group acting $2$-transitively on a set~$X$. Suppose that for some $x\in X$ (equivalently, for all $x\in X$), the action of the stabilizer $H=\stab_G(x)$ on $X\setminus\set{x}$ is mixable. Then the action of $G$ on $X$ is mixable, and
  \[
    \mixlen(G,X) \le \mixlen(H,X\setminus\set{x}) + 1.
  \]
\end{thm}

\begin{proof}
  We mix the action of $G$ on $X$ as follows. Fix $a\in G\setminus H$, and let~$\eps\sim\Ber(p)$ for $p=1-\frac{1}{\left|X\right|}$. Let $\boldsymbol{h}$ be a random subproduct in $H$ that mixes the action of $H$ on $X\setminus\set{x}$ with respect to $ax$ (i.e., $\boldsymbol{h}ax$ is uniform on $X\setminus\set{x}$). We claim that $\boldsymbol{g}=\boldsymbol{h}a^{\eps}$ mixes the action of $G$ on~$X$, i.e., $\boldsymbol{g}x$ is uniform on $X$.

  Indeed, $\boldsymbol{g}x=x$ if and only if $\eps=0$ (since otherwise $\boldsymbol{g}x=\boldsymbol{h}ax\ne x$), hence $\Pr(\boldsymbol{g}x=x)=\frac{1}{\left|X\right|}$. For any $x\ne x'\in X$, if $\boldsymbol{g}x=x'$ then we must have $\eps=1$, hence
  \begin{align*}
    \Pr(\boldsymbol{g}x=x') & = \Pr(\boldsymbol{g}x=x'\suchthat \eps=1)\Pr(\eps=1) \\
    & = \Pr(\boldsymbol{h}(ax)=x'\suchthat \eps=1)\Pr(\eps=1) \\
    & = \frac{1}{\left|X\right|-1}\cdot\frac{\left|X\right|-1}{\left|X\right|}=\frac{1}{\left|X\right|}.
  \end{align*}
\end{proof}

In terms of mixability of the group, we may deduce:
\begin{cor}\label{cor:2-transitive}
  Let $G$ be a finite group acting $2$-transitively on a set~$X$. Suppose that for some $x\in X$ (equivalently, for all $x\in X$), the stabilizer $H=\stab_G(x)$ is mixable. Then $G$ is mixable, and
  \[
    \mixlen(G) \le \mixlen(H) + \mixlen(H,X\setminus\set{x}) + 1.
  \]
\end{cor}

\begin{proof}
    The action of $H$ on $X\setminus\set{x}$ is mixable since $H$ is mixable, hence by \Tref{thm:2-transitive-act} the action of $G$ on $X$ is mixable as well with
    \[
        \mixlen(G,X) \le \mixlen(H,X\setminus\set{x}) + 1.
    \]
    Recall also that since the action is transitive, the action of $G$ on $X$ is isomorphic to the action of $G$ on $G/H$, hence $\mixlen(G,G/H) = \mixlen(G,X)$. The corollary now follows from and \Pref{prop:H-and-G/H}.
\end{proof}

This is utilized in the next section to prove the mixability of several families of groups.

Using induction, we can also write a similar criterion for higher transitivity:
\begin{cor}
  Let $G$ be a finite group that acts $k$-transitively on a set~$X$ for some $k\ge 2$. Let $G_{k-1}$ be the stabilizer of some distinct points $x_1,\dots,x_k\in X$. Recall that by $k$-transitivity, the action of $G_{k-1}$ on $X\setminus\set{x_1,\dots,x_{k-1}}$ is transitive.
  \begin{enumerate}
      \item If the action of $G_{k-1}$ on $X\setminus\set{x_1,\dots,x_{k-1}}$ is mixable, then the action of $G$ on $X$ is mixable. Consequently,
      \item If $G_{k-1}$ is mixable then $G$ is mixable.
  \end{enumerate}
\end{cor}

As another corollary, we show that certain semidirect products are mixable:

\begin{prop}\label{prop:mix-spe-semi}
    Let $G=N\rtimes Q$ be a semidirect product. Assume that $Q$ is mixable, and that the action of $Q$~on $N$ has exactly two orbits. Then $G$ is mixable.
\end{prop}

\begin{proof}
    We will show that, under the above assumptions, the action of~$G$ on $G/Q$ is $2$-transitive. Since the point stabilizer of this action is~$Q$, which is mixable by assumption, it will follow from \Cref{cor:2-transitive} that~$G$ is mixable.

    Fix $e\ne a\in N$. We claim that $G=Q\cup QaQ$. Indeed, let $g\in G$, and write $g=bh$ for $b\in N$ and $h\in Q$. If $b=e$, then $g=h\in Q$, and we are done. Otherwise, by assumption there is some $h'\in Q$ such that $h'a(h')^{-1}=b$, and thus $g=bh=h'a(h')^{-1}h\in QaQ$.

    To prove that the action of $G$ on $G/Q$ is $2$-transitive, it is enough to prove that the point stabilizer $Q$ of the coset $Q$ acts transitively on $(G/Q)\setminus\set{Q}$. But this follows directly from $G=Q\cup QaQ$, which shows that $Q$ maps $aQ$ to $(G/Q)\setminus\set{Q}$.
\end{proof}

\begin{rem}
The conditions in \Pref{prop:mix-spe-semi} imply that $N$ is elementary abelian.

Indeed, the transitive action on the nontrivial elements means all elements have the same order, which thus must be a prime $p$; as a $p$-group $N$  has a nontrivial center, but then all elements are central and~$N$ is a vector space over $\Z_p$.
\end{rem}

\section{Matrix groups}\label{sec:10}

We turn to study the mixability of the matrix groups $\PSL$, $\SL$, $\PGL$ and $\GL$. We begin by stating the implications between the mixability of these groups.

\begin{prop}\label{prop:matrix-rels}
    Let $q$ be a prime-power and let $d\ge 1$.
    \begin{enumerate}
        \item If $\SL[d](\F_q)$ is mixable, then $\PSL[d](\F_q)$ is mixable.
        \item If $\PGL[d](\F_q)$ is mixable, then $\gcd(q-1,d)$ is a power of $2$.
        \item If $\GL[d](\F_q)$ is mixable, then $\PGL[d](\F_q)$ is mixable and $q-1$ is a power of $2$.
        \item If $\PSL[d](\F_q)$ is mixable and $\gcd(q-1,d)$ is a power of $2$, then $\SL[d](\F_q)$ and $\PGL[d](\F_q)$ are mixable.
        \item If $\PSL[d](\F_q)$ or $\PGL[d](\F_q)$ are mixable, and $q-1$ is a power of~$2$, then $\GL[d](\F_q)$ is mixable.
    \end{enumerate}
\end{prop}

For any $d\ge 1$, write $\mu_d(\F_q)=\set{\alpha\in\F_q\suchthat\alpha^d=1}$, which is a subgroup of $\mul{\F_q}$ of order $\gcd(q-1,d)$.

\begin{proof}
    We use the classical commuting diagram with exact rows and columns:
    \[
        \begin{tikzcd}
            &  & 1 \arrow[d] & 1 \arrow[d] \\
            1 \arrow[r] & \mu_d(\F_q) \arrow[r] & \SL[d](\F_q) \arrow[r]\arrow[d] & \PSL[d](\F_q) \arrow[r]\arrow[d] & 1 \\
            1 \arrow[r] & \mul{\F_q} \arrow[r] & \GL[d](\F_q) \arrow[r]\arrow[d] & \PGL[d](\F_q) \arrow[r]\arrow[d] & 1 \\
            & & \mul{\F_q} \arrow[d] & \mul{\F_q} / (\mul{\F_q})^d \arrow[d] &  \\
            &  & 1 & 1
        \end{tikzcd}
    \]
    Note also that $\mu_d(\F_q)$ and $\mul{\F_q} / (\mul{\F_q})^d$ are both cyclic groups of order $\gcd(q-1,d)$, and are thus mixable if and only if $\gcd(q-1,d)$ is a power of $2$. Therefore, the first three implications follow since the class of mixable groups is closed under projections, and the last two implications follow from \Cref{cor:N-and-G/N}.
\end{proof}

In some cases, we can prove mixability explicitly:

\begin{thm}\label{thm:PSL2-F2n}
    For all $n\ge 1$, the group $\PSL[2](\F_{2^n})$ is mixable.
\end{thm}

\begin{proof}
    Consider the action of $\SL[2](\F_{2^n})$ on the projective line $\mathbb{P}^1(\F_{2^n})$. This action is $2$-transitive, so the induced action on pairs of projective points $X=\set{(x,y)\in(\mathbb{P}^1(\F_{2^n}))^2\suchthat x\ne y}$ is transitive. We write the points of $\mathbb{P}^1(\F_{2^n})$ as $[a:b]$ for $a,b\in\F_{2^n}$ such that $(a,b)\ne(0,0)$, so that $\mathbb{P}^1(\F_{2^n})=\set{[a:1]\suchthat a\in\F_{2^n}}\cup\set{[1:0]}$. We will prove that the action of $\SL[2](\F_{2^n})$ on $X$ is mixable, and that the stabilizer of this action is contained in a mixable subgroup of $\SL[2](\F_{2^n})$, which together imply the mixability of $\SL[2](\F_{2^n})$.

    We first prove that the action of $\SL[2](\F_{2^n})$ on $X$ is mixable. Recall that $\F_{2^n}$ is embedded, as an additive group, in $\SL[2](\F_{2^n})$, via the matrices $\begin{pmatrix}1&\alpha\\0&1\end{pmatrix}$. As $\F_{2^n}$ is mixable, there is a random subproduct $\boldsymbol{h}$ of elements of this form such that $\boldsymbol{h}=\begin{pmatrix}1&\boldsymbol{\alpha}\\0&1\end{pmatrix}$ where $\boldsymbol{\alpha}$ is uniform on $\F_{2^n}$. We take an independent copy $\boldsymbol{h}'=\begin{pmatrix}1&\boldsymbol{\alpha}'\\0&1\end{pmatrix}$ of $\boldsymbol{h}$.

    Let $\eps\sim\Ber(1-\frac{1}{2^n+1})$ be a Bernoulli random variable, independent of $\boldsymbol{h}$ and of $\boldsymbol{h}'$, and consider the random subproduct
    \[
        \boldsymbol{g} = \boldsymbol{h}\begin{pmatrix}0&-1\\1&0\end{pmatrix}^{\eps}\boldsymbol{h}.
    \]
    We will show that $\boldsymbol{g}$ mixes the action of $\SL[2](\F_{2^n})$ on $X$ with respect to the starting point $([1:0,0:1])$.

    Indeed, note that
    \[
        \boldsymbol{g} = \begin{cases}\begin{pmatrix}1&\boldsymbol{\alpha}+\boldsymbol{\alpha}'\\0&1\end{pmatrix}, & \text{if }\eps=0\\ \begin{pmatrix}\boldsymbol{\alpha}&\boldsymbol{\alpha}\boldsymbol{\alpha}'-1\\1&\boldsymbol{\alpha}'\end{pmatrix}, & \text{if }\eps=1.\end{cases}
    \]
    It follows that
    \[
        (\boldsymbol{g}[1:0],\boldsymbol{g}[0:1]) = \begin{cases}([1:0],[\boldsymbol{\alpha}+\boldsymbol{\alpha}':1]) , & \text{if }\eps=0 \\ ([\boldsymbol{\alpha}:1],[\boldsymbol{\alpha}\boldsymbol{\alpha}'-1:\boldsymbol{\alpha}']),&\text{if } \eps=1,\end{cases}
    \]
    which is distributed uniformly on $X$. This proves that the action of $\PSL[2](\F_{2^n})$ is mixable.

    The stabilizer of $([1:0],[0:1])$ under this action is the subgroup
    \[
        H = \set{\begin{pmatrix}\alpha&0\\0&\alpha^{-1}\end{pmatrix} \suchthat \alpha\in\mul{\F_{2^n}}}\cong\mul{\F_{2^n}},
    \]
    which is not mixable. However, consider the subgroup
    \[
        K = \sg{H,\begin{pmatrix}0&-1\\1&0\end{pmatrix}}.
    \]
    Since $\F_{2^n}$ is a cyclic group, one easily checks that $K\cong\Dih{2^n-1}$, and is thus mixable by \Tref{thm:coxeter}.

    Since $H$ is the stabilizer of the above action and the action is mixable, it follows that the action of $\SL[2](\F_{2^n})$ on $\SL[2](\F_{2^n})/H$ is mixable. Therefore, by \Pref{prop:G/K-and-G/H}, the action of $\SL[2](\F_{2^n})$ on $\SL[2](\F_{2^n})/K$ is also mixable. \Pref{prop:H-and-G/H} shows that $\SL[2](\F_{2^n})$ is mixable, concluding the proof.
\end{proof}

We remark that while $\PSL[2](\F_{2^n})=\SL[2](\F_{2^n})=\PGL[2](\F_{2^n})$ is mixable, the group $\GL[2](\F_{2^n})$ is not mixable for all $n\ge 2$ by \Pref{prop:matrix-rels}.

\section{Finite simple groups via $2$-transitivity}\label{sec:11}

In this section, we study the mixability of finite simple groups that have a $2$-transitive action.

\subsection{Matrix groups}

The first family of examples we consider are the special linear groups $\PSL[d](\F_q)$, which act $2$-transitively on the $(d-1)$-dimensional projective space $\mathbb{P}^{d-1}(\F_q)$. We show that in several cases, one can apply \Cref{cor:2-transitive} to prove their mixability.

We begin by computing the stabilizer of this action.

\begin{lem}\label{lem:PSL-stab}
    The stabilizer in the action of $\PSL[d](\F_q)$ on $\mathbb{P}^{d-1}(\F_q)$ is isomorphic to the semidirect product
    \[
        \F_q^{d-1}\rtimes(\GL[d-1](\F_q)/\mu_d(\F_q)I),
    \]
    where $\GL[d-1](\F_q)/\mu_d(\F_q)I$ acts on $\F_q^{d-1}$ by $v^A=(\det A)Av$.
\end{lem}

\begin{proof}
    Identify $\mathbb{P}^{d-1}(\F_q)$ with the set of $1$-dimensional subspaces of $\F_q^d$. Considering first the action of $\SL[d](\F_q)$ on $\mathbb{P}^{d-1}(\F_q)$, the stabilizer of $\F_q e_d$ is the subgroup
    \[
        H=\set{\begin{pmatrix}A & v \\ 0 & (\det A)^{-1}\end{pmatrix}\suchthat A\in\GL[d-1](\F_q),v\in\F_q^{d-1}}.
    \]
    This subgroup decomposes into the semidirect product induced by the action of $\GL[d-1](\F_q)$ on $\F_q^{d-1}$ given by $v^A=(\det A)Av$.

    Recall that $\PSL[d](\F_q)$ is defined as the quotient of $\SL[d](\F_q)$ by its center, which is $\mu_d(\F_q)I$. As $\mu_d(\F_q)I\le H$, the stabilizer of $\F_q e_d$ under the action of $\PSL[d](\F_q)$ is
    \[
        H/\mu_d(\F_q)I \cong \F_q^{d-1}\rtimes(\GL[d-1](\F_q)/\mu_d(\F_q)I).
    \]
    Note that the latter semidirect product decomposition makes sense since $\mu_d(\F_q)I$ acts trivially on $\F_q^{d-1}$ under the above action.
\end{proof}

We can now state an inductive claim, showing when one can lift the mixability of $\PSL[d-1](\F_q)$ to the mixability of $\PSL[d](\F_q)$.

\begin{prop}\label{prop:PSLd-step}
    Let $q$ be a prime-power and let $d\ge 2$. Assume that $\PSL[d-1](\F_q)$ is mixable. If $\frac{q-1}{\gcd(q-1,d)}$ is a power of $2$, then $\PSL[d](\F_q)$ is mixable.
\end{prop}

\begin{proof}
    Since $\gcd(q-1,d-1)$ and $\gcd(q-1,d)$ are coprime, it follows that $\gcd(q-1,d-1)|\frac{q-1}{\gcd(q-1,d)}$, so $\gcd(q-1,d-1)$ is a power of $2$. Therefore, the group $\PGL[d-1](\F_q)$ is mixable by \Pref{prop:matrix-rels}.

    We use the $2$-transitive action of $\PSL[d](\F_q)$ on $\mathbb{P}^{d-1}(\F_q)$. By \Lref{lem:PSL-stab}, the stabilizer of this action can be written as a semidirect product $H=\F_q^{d-1}\rtimes(\GL[d-1](\F_q)/\mu_d(\F_q)I)$, where $\GL[d-1](\F_q)/\mu_d(\F_q)I$ acts on $\F_q^{d-1}$ by $v^A=(\det A)Av$. We prove that $H$ is mixable, which will show that $\PSL[d](\F_q)$ is mixable by \Cref{cor:2-transitive}.

    We first check that $\GL[d-1](\F_q)/\mu_d(\F_q)I$ is mixable. Consider the short exact sequence
    \[
    \begin{tikzcd}[column sep=small]
        1 \arrow[r] & \mul{\F_q}/\mu_d(\F_q) \arrow[r] & \GL[d-1](\F_q)/\mu_d(\F_q)I \arrow[r] & {\PGL[d-1](\F_q)} \arrow[r] & 1.
    \end{tikzcd}
    \]
    The group $\mul{\F_q}/\mu_d(\F_q)$ is of order $\frac{q-1}{\gcd(q-1,d)}$, and thus it is mixable as a $2$-group. Since $\PGL[d-1](\F_q)$ is mixable as well, it follows that $\GL[d-1](\F_q)/\mu_d(\F_q)I$ is mixable. Next, the action of $\GL[d-1](\F_q)/\mu_d(\F_q)I$ on $\F_q^{d-1}\setminus\set{0}$ is transitive, since $\SL[d-1](\F_q)\hookrightarrow\GL[d-1](\F_q)/\mu_d(\F_q)I$ acts transitively on $\F_q^{d-1}\setminus\set{0}$. Therefore $H$ is mixable by \Pref{prop:mix-spe-semi}.
\end{proof}

For some values of $q$, this is enough to prove mixability for all $d$:

\begin{cor}\label{cor:matrix-all-d}
    If $q-1$ is a power of $2$ (including the case $q=2$), then $\PSL[d](\F_q)$, $\SL[d](\F_q)$ and $\PGL[d](\F_q)$ are mixable for all $d\ge 1$.
\end{cor}

\begin{rem}
We identify these prime-powers $q$ such that $q-1$ is a power of $2$. For a prime-power $q$, $q-1$ is a power of $2$ if and only if $q = 2$, $q = 9$ or $q$ is a Fermat prime (\eg\ $3$, $5$, $17$, $257$, $65337$, and conjecturally no other prime).

This follows from the proven Catalan conjecture (that the only solution to $x^a-y^b=1$ for $a,b>1$ is $3^2-2^3=1$) \cite{Preda04}.
\end{rem}

\begin{cor}\label{cor:GLd}
  The general linear group $\GL[d](\F_q)$ is mixable if and only if $q-1$ is a power of $2$.
\end{cor}

This shows, for instance, that $\PSL[2](\F_7)\cong\GL[3](\F_2)$ is mixable, which is also not covered by \Pref{prop:PSLd-step}.

\begin{proof}
  Recall that $\GL[d](\F_q)/\SL[d](\F_q)\cong\mul{\F_q}$. By the above proposition and \Cref{cor:abelian}, both $\SL[d](\F_q)$ and $\mul{\F_q}$ are mixable if and only if $q-1$ is a power of $2$, proving the corollary.
\end{proof}

\begin{rem}\label{rem:PSL-not-nec}
    \Pref{prop:PSLd-step} does not provide a necessary criterion for the mixability of $\PSL[d](\F_q)$. Indeed, $\PSL[2](\F_4)\cong\Alt{5}$ is mixable even though $\frac{4-1}{\gcd(4-1,2)}=3$ is not a power of $2$.

    Knowing that $\PSL[2](\F_4)$ is mixable, \Pref{prop:PSLd-step} tells us that $\PSL[3](\F_4)$ is mixable since $\frac{4-1}{\gcd(4-1,3)}=1$, but does not show whether $\PSL[4](\F_4)$ is mixable or not.
\end{rem}

\subsection{Sporadic simple groups}\label{sec:12}

We next utilize our machinery to study the mixability of some $2$-transitive sporadic simple groups -- the Mathieu groups and the Higman--Sims group.

A \textbf{Steiner system} $S(t,k,v)$ is a set $X$ of $v$ points, and a collection of subsets of $X$ of size $k$ (called \textbf{blocks}), such that any $t$ points of $X$ are in exactly one of the blocks. The \textbf{Mathieu groups} were introduced by Mathieu \cite{Mathieu61,Mathieu73} as the automorphism groups of certain Steiner systems, and were the first examples of sporadic simple groups to be discovered. These groups act multiply transitive on the set of points of the appropriate Steiner system. See Table~\ref{tab:mathieu-groups} for the list of the five simple Mathieu groups, their corresponding Steiner system, and their transitivity order.

\begin{table}[ht]
    \centering
    \begin{tabular}{|c|c|c|c|c|}
      \hline
      Group & Steiner system & Transitivity & Order & Point stabilizer \\\hline
      $M_{11}$ & $S(4,5,11)$ & $4$ & $2^4\cdot 3^2\cdot 5\cdot 11$ & $\Alt{6}$ \\ \hline
      $M_{12}$ & $S(5,6,12)$ & $5$ & $2^6\cdot 3^3\cdot 5\cdot 11$ & $M_{11}$ \\ \hline
      $M_{22}$ & $S(3,6,22)$ & $3$ & $2^7\cdot 3^2\cdot 5\cdot 7\cdot 11$ & $\PSL[3](\F_4)$ \\ \hline
      $M_{23}$ & $S(4,7,23)$ & $4$ & $2^7\cdot 3^2\cdot 5\cdot 7\cdot 11\cdot 23$ & $M_{22}$ \\ \hline
      $M_{24}$ & $S(5,8,24)$ & $5$ & $2^{10}\cdot 3^3\cdot 5\cdot 7\cdot 11\cdot 23$ & $M_{23}$ \\ \hline
    \end{tabular}

    \caption{Definition and properties of the Mathieu groups.}
    \label{tab:mathieu-groups}
\end{table}

\begin{cor}\label{cor:mathieu}
  The five Mathieu groups $M_{11},M_{12},M_{22},M_{23},M_{24}$ are mixable.
\end{cor}

\begin{proof}
  As mentioned above, all five Mathieu groups act $2$-transitively on the set of points of their corresponding Steiner system. We apply \Cref{cor:2-transitive} to all five groups: $\Alt{6}$ and $\PSL[3](\F_4)$ are mixable by \Cref{cor:alternating} and \Rref{rem:PSL-not-nec} respectively, hence $M_{11}$ and $M_{22}$ are mixable; this shows that $M_{12}$ and $M_{23}$ are mixable as well; and this in turn proves the mixability of $M_{24}$.
\end{proof}

Another example of a sporadic $2$-transitive group is the \textbf{Higman--Sims group} $\HS$, found by Higman and Sims \cite{HigmanSims68}. This group is a subgroup of index $2$ in the automorphism group of the Higman--Sims graph (a certain $22$-regular graph with $100$ vertices), and its action on the graph is $2$-transitive. Since the point stabilizer in this action is isomorphic to the Mathieu group $M_{22}$, we may apply \Cref{cor:2-transitive} and get:

\begin{cor}\label{cor:higman-sims}
The Higman--Sims group $\HS$ is mixable.
\end{cor}

\bibliographystyle{plain}
\bibliography{mix_refs}

\end{document}